\documentclass[a4paper,11pt]{article}
\usepackage{amsmath, amsthm, amssymb, bbm}
\usepackage{enumitem}
\usepackage{url}

\usepackage{mathtools}
\usepackage[top=1.3in, bottom=1.3in, left=1.2in, right=1.2in]{geometry}
\renewenvironment{proof}{\noindent\textbf{Proof:\ }}{\hspace*{\fill}$\square$\medskip}
\newtheorem{theorem}{Theorem}
\newtheorem{definition}{Definition}[section]
\newtheorem{lemma}[definition]{Lemma}
\newtheorem{proposition}[definition]{Proposition}
\newtheorem{remark}[definition]{Remark}
\newcommand{\cG}{\mathcal{G}}
\newcommand{\cQ}{\mathcal{Q}}
\newcommand{\cR}{\mathcal{R}}
\newcommand{\cS}{\mathcal{S}}
\newcommand{\cT}{\mathcal{T}}
\newcommand{\cU}{\mathcal{U}}
\newcommand{\cMT}{\mathcal{MT}}
\newcommand{\N}{\mathbb{N}}
\newcommand{\R}{\mathbb{R}}
\newcommand{\bS}{\mathbb{S}}
\newcommand{\bT}{\mathbb{T}}
\newcommand{\Z}{\mathbb{Z}}

\newcommand{\nin}{\notin}
\newcommand{\eps}{\varepsilon}
\newcommand{\norm}[1]{\left\|#1\right\|}
\newcommand{\pd}{\partial}
\newcommand{\dd}{\,\mathrm{d}}
\newcommand{\1}{\mathbbm{1}}
\DeclareMathOperator*{\esssup}{ess\,sup}

\numberwithin{equation}{section}

\begin{document}
\title{Long Range Particle Dynamics and the Linear Boltzmann Equation }
 \author{Matthew Egginton\footnote{Mathematics Institute, University of Warwick. m.egginton@warwick.ac.uk. 
Funding provided by MasDoc DTC grant number EP/H023364/1 provided by EPSRC.}  and Florian Theil\footnote{Mathematics Institute, University of Warwick. f.theil@warwick.ac.uk}}
 \maketitle

 \begin{abstract}
 This paper gives the first full proof of the justification of the linear Boltzmann equation from an underlying long range particle evolution. We suppose that a tagged particle is interacting with a background via a two body potential that is decaying faster than $ C\exp\left(-C|x|^{\frac{3}{2}}\right) $, and that the background is initially distributed according to a function in $L^1((1+|v|^2)\dd v)$ in velocity and uniformly in space. Under finite mass and energy assumptions on the initial density, the tagged particle density converges weak-$\star$ in $L^\infty$ to a solution of the linear Boltzmann equation.
 
 The proof uses estimates on two body scattering and on the relationship between long range dynamics and dynamics with a truncated interaction potential to explicitly estimate the error between densities for long and short range dynamics.
 To compare the difference between the short range dynamics and the linear Boltzmann equation, we use a tree based structure to encode the collisional history of the tagged particle.
 \end{abstract}
 
 \section{Introduction}\label{sec:intro}
 
 One of the main purposes of kinetic theory is to provide a justification of the macroscopic laws of motion for gaseous fluids, where the main effect on the density of the fluid is given by the statistical impact of collisions. A famous open problem in this area is the justification of the non-linear Boltzmann equation from an underlying particle dynamics model. The first proofs of this for short times were given by Lanford and King \cite{lanford1975, king1975} and use the BBGKY hierarchy. These proofs are valid under the assumption that the interactions between the particles are short range in nature. This proof has been completed in \cite{gallagher2013} and extended in \cite{pulvirenti2014}. Furthermore, \cite{bodineau2015} uses the BBGKY hierarchy to derive the linear Boltzmann equation as a perturbation from equilibrium of the fully non-linear Boltzmann equation for hard sphere interactions. Other methods have been proposed for the gainless Boltzmann equation, as in \cite{matthies2012}.
 
 However, most physically justifiable interaction potentials are long range in nature, and even in the origins of the Boltzmann equation \cite{maxwell1866, boltzmann1970} the interactions were assumed to be long range. With the exception of \cite{Ayi2017} where a near equilibrium assumption is made, there are no  results justifying the Boltzmann equation with a long range interaction in the mathematical literature. There are several results relating to the existence of solutions for the Boltzmann equation with long range potential, for example \cite{alexandre2002, gressman2010}, as well as results for other kinetic theory models, for example in the study of the Vlasov-Poisson system, where there are existence theory results as for example \cite{pfaffelmoser1992, crippa2017} for a Coulomb force field.
 
A study of justification in this setting poses two problems; firstly the non-linearity of the Boltzmann equation and secondly the long range nature of the interaction. The first can be removed by considering a simpler particle interaction which enables the study of linear equations. An example of a simpler particle system is given by the Lorentz gas \cite{lorentz1905} or the Rayleigh gas, where one assumes the only interactions present are between a unique particle of one species and a collection of particles of another species that does not self interact, and we use this system. In this setting one only needs to consider the density of the unique particle species, and the macroscopic evolution is given by the linear Boltzmann equation.

The Lorentz gas is well studied in this context, and it has been shown that the linear Boltzmann equation is the low density limit for such a particle system, where the background scatterers are fixed and randomly placed and the interactions are hard spheres \cite{Gallavotti72, vanBeijeren1980, lebowitz1982, boldrighini1983} or short range potential interactions \cite{Spohn1978}. However the linear Boltzmann equation can fail as the macroscopic limit if there are non-random periodic scatterers, see for example \cite{Golse2008,marklof2010}. With regards to a Rayleigh gas particle system, where the background can move, the linear Boltzmann equation has been proven to be the macroscopic limit with hard sphere interactions in \cite{Matthies2018, matthies2017} for arbitrarily long times.

The second problem introduced above was the inclusion of long range interactions. The first difficulty with the use of long range forces is that the typical studies of the linear Boltzmann equation \cite{dautray1993, allaire2015} are only valid in the case of compact interactions. Furthermore, from a particle dynamics perspective, one no longer has a well defined notion of a collision, since the particles interact for all time. This adds additional difficulty when the primary operation in the evolution equations in kinetic theory are two body collisions. By  removing the long range part of the interaction with a regularising parameter, in the setting of the Lorentz gas, the paper \cite{desvillettes1999} demonstrated that solutions of a cut off particle evolution converge to an uncut off linear Boltzmann equation. Our analysis goes further than this, to say that if the potential decays sufficiently fast, then convergence holds on the level of the long range processes. 
 
 The purpose of this paper is to give the first full proof of the justification that the linear Boltzmann equation is the low density limit for a long range Rayleigh gas particle system. We start by introducing the particle model, before stating the main theorem.
 
Suppose that one distributes $N\in \N$ particles independently and identically according to a density which is uniform in the 3 dimensional torus $\bT^3$ and according to a density  $g$ with finite mass and energy on the velocity space $ \R^3$ which we call the background particles, and distribute one particle independently from this background according to a probability density $f_0$ on the phase space $\cU= \bT^3 \times \R^3$, which we call the tagged particle. On the position space $\bT^3$ we impose periodic boundary conditions. We restrict ourselves throughout to three dimensions to ease notation, but in principle the method works for any dimension greater than or equal to two.
 
 Suppose that, for microscopic spatial scale $\eps>0$, we allow the positions of these particles to evolve via the equations,  
 \begin{equation}\label{eq:partdens}
  \begin{aligned}
   \dot x(t) &=v(t) , \hspace{1cm} \dot v(t) = -\frac{1}{\eps}\sum_{j=1}^{N}\nabla \phi\left(\frac{x(t)-x_j(t)}{\eps}\right)\\
   \dot x_i(t)&=v_i(t), \hspace{1cm} \dot v_i(t) =0\\
 \end{aligned}
 \end{equation}
where the background is indexed from $i=1,\dots,N$, and the function $\phi\colon \R^3 \to \R$ is the interaction potential. Since the evolution is on $\bT^3$, by $x-x_i$ we mean the vector $x-(x_i+k)$ for $k\in \Z^3$ for which the Euclidean distance is the minimum.

We denote by $f^\eps\colon [0,T]\times \cU\to \R$ the phase space density for the tagged particle. We remark here that the majority of the interactions between the tagged particle and the background are grazing, where the distance between the particles is large.   One should then expect these interactions individually to deviate the velocity of the tagged particle by a small amount, and also the contribution from all grazing collisions should affect the distribution $f^\eps$ in a small but quantifiable manner. This is made  precise in Section~\ref{sec:grazcol}.

The macroscopic evolution of the tagged particle density on phase space $\cU$ and time interval $[0,T]$ is given by a weak solution of the linear Boltzmann equation 
\begin{equation*}
 \pd_t f+v\cdot \nabla _x f = L(f) ,
\end{equation*}
where the linear collision operator is given by
\begin{equation}\label{eq:strongcollop}
 L(f):= \int_{\R^3} \int_\cS \left(f'\,g'_\star-f\, g_\star\right) |v_\star-v|\dd S \dd v_\star,
\end{equation}
where we make precise the meaning of weak solution in the statement of Theorem~\ref{thm:main}. The shorthand $g_\star=g(v_\star)$ is used for evaluation of density at the velocity of the colliding particle and $f'=f(v')$, $g'_\star =g(v'_\star)$  are used to represent evaluation of the densities at the pre-collisional velocities in a two body interaction. Furthermore, 
\begin{equation}\label{eq:planeS}
 \cS= \{w\in \R^3: w\cdot (v_\star-v)=x\cdot (v_\star-v)\}
\end{equation}
is the parameter space for possible interactions. The pre-collisional velocities are determined by the potential $\phi$ from the underlying two body particle dynamics. There are various ways of writing \eqref{eq:strongcollop}, and our notation originates in \cite{truesdell1980}.

The use of weak solutions is a natural consequence of the long range interaction, where the grazing collisions ensure that one cannot make sense of the strong form given above due to the singularity encompassed in integration over the tail of $\phi$ in the plane $\cS$. The integral however does converge in a weak sense for arbitrary $f\in L^1$.  We comment in more detail in Section~\ref{sec:slns}. Furthermore, this equation is the Fokker-Planck equation for a Markov process, and this structure will be exploited in later sections when comparing the particle and Boltzmann evolutions.

The analysis of later sections allows for the comparison of the two scales only for potentials which satisfy the following conditions, as well as a decay assumption on the potential, which is stated in equation \eqref{eq:potdecay} in Theorem~\ref{thm:main}.

 \begin{definition}\label{def:admisspot}
 A potential $\phi\colon \R^3\to \R$ is an {\bf admissible long range potential} if
   \begin{enumerate} 
\item $\phi$ is radial, namely there is a function $\psi\in C^2 (0,\infty)$ such that $\phi(x)=\psi(|x|)$,
\item $\psi$ is strictly decreasing,
\item $\lim\limits_{\rho\to \infty}\psi(\rho)=0$ and $\lim\limits_{\rho\to 0}\psi(\rho)=\infty$,
\item There is a $\rho_1>0$ such that for $\rho\in (0,\rho_1)$, we have $\frac{\dd}{\dd \rho}\psi(\rho)+ \psi(\rho)\leq 0$.
\end{enumerate}
\end{definition}

This definition should be compared with the definition of the potential in \cite{gallagher2013},  as this book is the inspiration for our conditions on the interaction potential. Assumption (1) delivers sufficient regularity to make sense of equations \eqref{eq:partdens}, and so that later estimates are well defined. Assumptions (2) and (3), together with the radial assumption in (1), should be thought of as conditions to ensure that the interaction is well defined, and in particular they ensure that the interaction is repulsive. Furthermore, the second limit in (3) ensures that singularities from coalescence are not present in the dynamics.

The motivation for using condition (4) is to ensure that one can estimate the scattering time for near collisions and provides control on the estimate in terms of the radius of the potential. To see that such an unbounded potential can indeed satisfy conditions (3) and (4) together, consider $\psi(\rho)=\rho^{-s}$ for $s>0$. Then $\frac{\dd}{\dd \rho}\psi(\rho)= -s\rho^{-s-1}=-s\rho^{-1}\rho^{-s}$ and so for $\rho<s$ we have the relationship described in (4).

The system is completed by specifying the assumptions on the initial density of the tagged particle, which are included in our main theorem.
 \begin{theorem}\label{thm:main}
  Let $T>0$, and suppose that $f_0\geq 0$ with $f_0\in L^1\cap L^\infty(\cU, (1+|v|^2)\dd x \dd v)$, and suppose that  $g\geq 0$ with
  \begin{equation}\label{eq:bckgrnd}
   \begin{aligned}
      \int_{\R^3}(1+|v|^2)\,g(v) \dd v &<\infty,\\ 
  \esssup_{v\in \R^3} (1+|v|^5)g(v)&<\infty.
   \end{aligned}
  \end{equation} 
   Suppose that $f^\eps$ is the phase space density of the tagged particle evolving via \eqref{eq:partdens} on the space $[0,T]\times \cU$ with initial density given by $f_0$, and background distributed according to $g$  with $\phi$ an admissible potential as in Definition~\ref{def:admisspot}, such that there are a $\rho_2>0$, constant $C>0$ and $\gamma>0$ such that the radial force satisfies
 \begin{equation}\label{eq:potdecay}
  -\frac{\dd}{\dd \rho}\psi(\rho)\leq Ce^{-C\rho^{\frac{3}{2}+\gamma}}
 \end{equation}
for all $\rho>\rho_2$. Then as $\eps\to0$ with $N\eps^{2}=1$ we have $f^\eps$ converges weak-$\star$ in $L^\infty$, up to a subsequence, to a weak solution $f$ of the linear Boltzmann equation associated to $\phi$ on $[0,T]\times \cU$, meaning that
 \begin{equation*}
  \int_\cU (1+|v|^2)\, f(t,x,v) \dd x \dd v <\infty
 \end{equation*}
for all $t\in [0,T]$ and that, for $h\in C^\infty_0([0,T)\times \cU)$, we have
   \begin{equation}\label{eq:weaksol}
    -\int_0^T\int_\cU  \left(\pd_t h +v\cdot \nabla_x h\right)\,f\dd x \dd v \dd t-\int_{\cU} f_0 \,h(0) \dd x \dd v= \int_0^T\langle L(f) ,\,h\rangle \dd t
   \end{equation}
   where the action of the collision operator on a test function $h$ is defined by
   \begin{equation*}
    \langle L(f) ,\,h\rangle=\int_\cU \int_{\R^3}\int_\cS (h'-h) \,f\, g_\star \,|v_\star-v|\dd S\dd v_\star \dd x \dd v,
   \end{equation*} 
   where $\cS$ is defined in \eqref{eq:planeS}   and $v'$ is specified in equation~\eqref{eq:scatop}.
 \end{theorem}
 
\begin{remark}\normalfont
\begin{enumerate}
 \item Evolution for the density of a tagged particle in a background is in general non-Markovian, and in the setting with a non-local interaction, every particle is always recolliding, and so one does not have a Markovian realisation of the tagged particle density. As shown in \cite{Matthies2018} for hard spheres, there is an associated Markovian evolution where one has removed recollisions. Truncation to obtain a short range interaction is used in \cite{crippa2017} to gain Markovian evolutions for a Vlasov-Poisson system, 
 and this motivates the use of a truncation parameter to create a short range Markovian system, to compare with the linear Boltzmann equation.
 \item In comparison to \cite{Ayi2017}, where one can show convergence in the near equilibrium regime for the force with decay of 
 \begin{equation*}
 -\frac{\dd}{\dd\rho} \psi(\rho) \leq e^{-e^{e^{\lambda\left(1+\rho^{2(d-1)}\right)}}}, 
 \end{equation*}
 our decay condition on the potential is very weak. It is unsurprising that we obtain a less restrictive decay assumption, as we have only one colliding particle which means we only need encode information on the collisions of the tagged particle, as opposed to encoding information on every particle.
 \item The paper \cite{desvillettes1999} considers a tagged particle moving in a fixed background, where the tagged particle and a background particle interact via a short range power law potential $\phi(x)= |x|^{-s}$ with $s>2$ and with cut off $|x|\leq \eps^{\gamma-1}$ for some $\gamma\in \left(\frac{15}{17},1\right)$. In this paper they show convergence of the tagged particle density to a solution of the long range linear Boltzmann equation for power law potential $|x|^{-s}$ for $s>2$. Note however that this result does not compare the particle systems for long and short range potentials, which is new analysis in our paper.
\end{enumerate}
\end{remark}

\subsection{Structure of the Proof}

The main objective is to show that the long range particle density $f^\eps$ converges to $f$ which is a solution of the linear Boltzmann equation for $\phi$. Due to the long range dynamics, the particle evolution is not Markovian, and so we introduce a regularisation parameter $R>0$ and a truncated short range potential $\phi^R$, which has support in $B_R$.
One should observe that this cut off is stronger than the Grad cut off from \cite{grad1958}, since we require a short range Markovian particle evolution as well as an integrable Boltzmann collision operator. 

This potential $\phi^R$ then enables one to define probability densities $f^{\eps,R}$ and $f^R$ corresponding to the short range dynamics and the linear Boltzmann equation associated to $\phi^R$, and we note that the particle dynamics then become Markovian, up to recollisions. We then compare the intermediate densities via
\begin{equation*}
 |f^\eps-f|\leq  |f^\eps-f^{\eps,R}|+ |f^{\eps,R}-f^R|+ |f^R-f|
\end{equation*}
and we desire estimates as the spatial distance $\eps \to 0$ with the regularisation parameter $R\to \infty$. The estimates used require $R= \eps^{-1/(3+\gamma)}$ where this exponent is necessary to ensure that estimates in comparing $f^{\eps,R}$ and $f^R$ decay to $0$ as $\eps\to 0$. In actuality we need only $R=\eps^{-\alpha}$ for some $0<\alpha \leq 1/(3+\gamma)$ but we use the specific form for simplicity.

To compare the short range densities $f^{\eps,R}$ and $f^R$ in Section~\ref{sec:MT}, we use the methodology of \cite{Matthies2018}, which enlarges the state space by deftly encoding the entire history of collisions of the tagged particle into a marked tree structure. This enlarging of the state space should be compared with the BBGKY hierarchy, where one enlarges the state space by considering all marginals of the particle dynamics.

We are then required to analyse the contribution from the grazing collisions on the short range particle evolution, which is performed in Section~\ref{sec:partdyn}. We identify the evolutions for which the long and short range particle dynamics have the same collisional structure, since in this situation, the difference between the two systems can be estimated explicitly. To conclude, we are then required to estimate the size of the set for which the long and short range evolutions do not exhibit the same collisions. This analysis is the bottleneck for forcing the decay of the potential to be exponential.

The final comparison in Section~\ref{sec:BEcomp} of the linear Boltzmann equations for potentials $\phi^R$ and $\phi$ uses compactness arguments to show that the solutions $f^R$ have a limit as $R\to \infty$, and then estimates on the collision operators to show that the limit satisfies equation~\eqref{eq:weaksol}. These estimates are similar in spirit to \cite{Ayi2017} and especially to \cite{desvillettes1999}. The latter compares the collision operators by comparing the Boltzmann kernels in terms of the deviation angle for inverse power law potentials decaying faster than $\rho^{-2}$. This is made tractable by the semi-explicit forms of the kernel for such potentials. For us, however, we use the estimates in Section~\ref{sec:scatter}  on the difference between the deviation angles to compare $L^R$ and $L$ explicitly.

 \section{Estimates for the Collision Operator} \label{sec:scatter}
 
 In the statement of the linear Boltzmann equation we introduced the pre-collisional velocities $v'$ and $v'_\star$, as well as the parameter space $\cS$, and we now describe their relation with the interaction potential $\phi$. This specification will then allow for the comparison between pre-collisional velocities for scattering with the potentials $\phi$ and $\phi^R$, as well as providing an estimate the scattering time for the short range potential $\phi^{R}$. The second purpose of this section is to prove existence and uniqueness of solutions for the linear Boltzmann equation with cut off potential $\phi^R$.
 
\subsection{Grazing Collisions and Estimates on Binary Interactions}\label{sec:grazcol}

The notion of grazing collision has not yet been made precise, and we take the following standard definition.
 
 \begin{definition}\label{def:grazcoll} Suppose for spatial scale $\eps$ that $x^\eps$ is the evolution of the tagged particle and $x_s$ the evolution of a background particle.
  A {\bf grazing collision} is an interaction for which 
  \begin{equation*}
   \frac{1}{\eps}\min|x^\eps-x_s|\geq R
  \end{equation*}
and a {\bf near collision} is an interaction with this distance strictly smaller than $R$.
 \end{definition}

 In order to demonstrate that these grazing collisions affect the dynamics weakly, we first specify the form of the cut off used to define $\phi^R$. Suppose that $\phi$ is an admissible long range potential as in Definition~\ref{def:admisspot}, and $R>0$. We then define the related short range potential $\phi^{R}$ by 
 \begin{equation*}
  \phi^{R}=\Lambda^R\phi
 \end{equation*}
with $\Lambda^R\in C^\infty(\R^3)$ a radial strictly decreasing function with $\Lambda^R(x)=1$ for $|x|\leq R-1$ and $\Lambda^R(x) =0 $ for $|x|\geq R$.

The parameters which describe a binary collision are the relative velocity $w=v_\star-v$, and the polar coordinates $(r,\zeta)$ of the plane $\cS$, which one recalls from equation~\eqref{eq:planeS} is
\begin{equation*}
 \cS= \{w\in \R^3: w\cdot (v_\star-v)=x\cdot (v_\star-v)\}.
\end{equation*}
 The distance $r$ corresponds to the minimum distance between the two particles without interaction, and $\zeta$ specifies the direction between them. The range of $r$ is $[0,\infty)$ for the potential $\phi$ and is $[0,R)$ for the potential $\phi^{R}$. The radial symmetry of the collision further means the description does not depend upon $\zeta$, although full details can be found in \cite[Ch.6]{truesdell1980}. 
 
A two body collision maps ingoing velocities $v,v_\star$ to outgoing velocities $v',v'_\star$ of the two particles. For any collision, the scattering map $\sigma(r,\zeta,v,v_\star)=(v',v'_\star)$ takes the general form
\begin{equation}\label{eq:scatop}
\begin{aligned}
    v'&=v+\left( w\cdot\nu(r,\zeta,w)\right)\nu(r,\zeta,w),\\
  v'_\star &=v_\star -\left( w\cdot\nu(r,\zeta,w)\right)\nu(r,\zeta,w)
\end{aligned}
\end{equation}
from conservations of momentum and energy, where the vector $\nu(r,\zeta,w)\in \bS^{2}$ depends upon the potential. The projection of $\nu$ onto the plane $\cS$ is given by
\begin{equation*}
 \nu\cdot (v_\star -v) = |v_\star-v|\,\sin \left(\frac{1}{2}\theta(r,v_\star-v)\right)
\end{equation*}
where $\theta(r,w)$ is called the deviation angle, which is given by the formula
 \begin{equation*}
\theta(r,w) = \pi-2\int_{\rho_\star}^{\infty} \frac{r \dd \rho}{\rho^2\sqrt{1-\frac{2\psi(\rho)}{|w|^2}-\frac{r^2}{\rho^2}}},
\end{equation*}
with $\rho_\star$ the largest root of the denominator. We remark that this integral does converge for all admissible long range potentials and for the related short range potentials. 

For the potential $\phi^{R}$, we add a superscript $R$ to the deviation angle, as well as to the pre-collisional velocities obtained from this in equation \eqref{eq:scatop}, and so we write $\theta^R$, $v^{\prime,R}$ and $v^{\prime,R}_\star$.

\begin{lemma}\label{lem:angdev}
 Suppose that $\phi$ is an admissible long range potential with the condition that there is a $\rho_2>0$ and $s>2$  such that for $\rho>\rho_2$ we have
 \begin{equation*}
\psi(\rho)\leq \rho^{- s} ,
 \end{equation*}
 and suppose that the relative velocity $|w|\geq \eta$ for some $\eta>0$. Then for $R $ such that $R-1/\eta>1+\rho_2$, we have
 \begin{equation}\label{eq:angerr}
  |\theta(r,w)-\theta^{R}(r,w)|\leq \begin{cases}\frac{C}{1+ \eta^2\, r^{s}}& \text{for } r>R-1-1/\eta\\
  &\\\frac{C\,\kappa(r,R)}{\eta^2} &\text{for } r<R-1-1/\eta,  \end{cases}
 \end{equation}
where the constants are independent of $r,R$ and $|w|$ and 
\begin{equation*}
\int_0^{R-1-1/\eta} r \,\kappa(r,R)\dd r =o(1)
\end{equation*}
as $R\to \infty$.
Furthermore, the scattering time for evolution under $\phi^{R}$ can be bounded by
 \begin{equation*}
  \tau_\star(r,w,R) \leq C\,\frac{R}{\eta}.
 \end{equation*}
\end{lemma}
\begin{remark}
 The key property of the estimate in \eqref{eq:angerr} is that the integral of the right hand side tending to zero implies that potentials with decay faster than $\rho^{-2}$ have their respective collision operators converging to each other, as will be seen in Section~\ref{sec:BEcomp}.
\end{remark}

\begin{proof}
 Firstly, for $r>\rho_2$, by extending results in \cite{desvillettes1999}, we obtain the estimate
 \begin{equation*}
  \theta(r,w)\leq \frac{C}{1+ \eta^2\, r^{s}}
 \end{equation*}
and, together with $\theta^{R}\ge 0$, one obtains $\theta -\theta^{R}\leq \theta $. Therefore we have 
\[                                                                                                                                                                     
\theta-\theta^{R}\leq\frac{C}{1+ \eta^2\, r^{s}}
\]
for $r>\rho_2$, and the choice of $R$ ensures that this holds in particular for $r>R-1-1/\eta$. 

 For $r<R-1-1/\eta$, by analysing the equation for $\rho_\star^R$, it is easy to show that $\rho_\star^R<R-1$. On the region $r<R-1-1/\eta$ one has $\phi^{R}=\phi$ and so we can then conclude that $\rho_\star=\rho_\star^{R}$.  We then split the difference $\theta -\theta^{R}$ into a difference corresponding to the long range nature of $\phi$, and an error which comes from the choice of the smooth cut off $\Lambda^R$. We consider these terms separately.

For the long range term, since $\psi^R=0$ for $\rho>R$, we obtain
\begin{equation*}
\begin{aligned}
  \int_{R}^\infty \frac{r \dd \rho}{\rho^2\sqrt{1-\frac{2\psi(\rho)}{|w|^2}-\frac{r^2}{\rho^2}}}&-\int_{R}^\infty \frac{r \dd \rho}{\rho^2\sqrt{1-\frac{r^2}{\rho^2}}}\\ =&\int_{R}^\infty \frac{r\frac{2\psi(\rho)}{|w|^2}\dd \rho}{\rho^2\sqrt{1-\frac{r^2}{\rho^2}}\sqrt{1-\frac{2\psi(\rho)}{|w|^2}-\frac{r^2}{\rho^2}}\left(\sqrt{1-\frac{r^2}{\rho^2}}+\sqrt{1-\frac{2\psi(\rho)}{|w|^2}-\frac{r^2}{\rho^2}}\right)}
\\\leq& \frac{C}{1-\frac{r^2}{R^2}}\frac{\sup\limits_{\rho\in(R,\infty)}\psi(\rho)}{|w|^2}\int_{R}^\infty \frac{r\dd \rho}{\rho^2\sqrt{1-\frac{r^2}{\rho^2}}}
 \end{aligned}
\end{equation*}
where we have used $C\left(1-\frac{r^2}{R^2}\right)^{-1}$ as an upper bound on the final three square roots in the denominator. 

For the cut off error term, by rearranging and bounding terms similarly to before, we obtain
\begin{equation*}
  \int_{R-1}^R\left(\frac{r}{\rho^2\sqrt{1-\frac{2\psi(\rho)}{|w|^2}-\frac{r^2}{\rho^2}}}-\frac{r}{\rho^2\sqrt{1-\frac{2\psi^{R}(\rho)}{|w|^2}-\frac{r^2}{\rho^2}}}\right) \dd \rho\leq\frac{C\,r\,\norm{(1-\Lambda^R)\phi}_{L^\infty}}{|w|^2R(R-1)\left(1-\frac{r^2}{(R-1)^2}\right)^{3/2}}.
\end{equation*}
Using the fact that $\norm{(1-\Lambda^R)\phi}_{L^\infty}\leq R^{-s}$, and then defining
\begin{equation}\label{eq:kappa}
 \kappa(r,R)= R^{-s} \left(\frac{1}{1-\frac{r^2}{R^2}}+\frac{r}{(R-1)\left(1-\frac{r^2}{(R-1)^2}\right)^{3/2}}\right)
\end{equation}
we observe that it satisfies the integral condition.

In order to gain a suitable bound on the scattering time, we split collisions into situations where the impact parameter is in the three regions $\left[\frac{R}{2},R\right]$, $\left[\frac{1}{2\sqrt2}\psi^{-1}\left(\frac{|w|^2}{4}\right),\frac{R}{2}\right]$ and in $\left[0,\frac{1}{2\sqrt2}\psi^{-1}\left(\frac{|w|^2}{4}\right)\right]$, where we use $\psi^{-1}$ instead of $(\psi^R)^{-1}$ to simplify notation. 
 
 By a simple extension of \cite[Lemma 1]{pulvirenti2014} to potentials supported in $B_R(0)$, we can bound the scattering time in the desired manner for $r\in \left[\frac{R}{2},R\right]$. Furthermore, the conditions required for admissible long range potentials, in particular (4), ensure that we can proceed as in the proof of \cite[Prop. 2]{Ayi2017}, and so one can show for $r\in \left[0,\frac{1}{2\sqrt2}\psi^{-1}\left(\frac{|w|^2}{4}\right)\right]$ the desired inequality.

We are thus left to analyse for $r\in\left[\frac{1}{2\sqrt2}\psi^{-1}\left(\frac{|w|^2}{4}\right),\frac{R}{2}\right]$. We have, for any parameter $r$ in this region, that 
\begin{equation*}
 \tau_\star(r,w,R) \leq \frac{\max_{r'} l(r')}{\min_{r'} w_\star(r')}
\end{equation*}
where $l(r)$ is the length of the particle path, and $w_\star(r)$ is the minimum velocity of the path for impact parameter $r$. The assumptions on $\phi$ ensure that we have that $l(r)\leq 2R$, and so we are left to provide a lower bound on $w_\star(r)$ over this region.

The point at which the particle has lowest relative velocity is the point at which it has maximal potential energy, which is the closest point, namely $\rho_\star$. Then we have from conservation of energy that
\begin{equation*}
 \frac{1}{2}|w|^2-\frac{1}{2}|w_\star|^2 =  \psi^{R}(\rho_\star)-\psi^{R}(R)=\psi^{R}(\rho_\star)
\end{equation*}
since $\psi^{R}(R)=0$ by assumption. Rearranging, we obtain
\begin{equation*}
 |w_\star|= \sqrt{|w|^2-2\psi^{R}(\rho_\star)},
\end{equation*}
and since from the equation for $\rho_\star$ we have
\begin{equation*}
 2\psi^{R}(\rho_\star)= |w|^2\left(1-\frac{r^2}{\rho_\star^2}\right),
\end{equation*}
combining the two equalities results in 
\begin{equation*}
 |w_\star|= \sqrt{|w|^2-2\psi^{R}(\rho_\star)}= |w|\sqrt{1-1+\frac{r^2}{\rho_\star^2}}=|w|\frac{r}{\rho_\star}.
\end{equation*}
We conclude by finding a lower bound on $r/\rho_\star$. On the interval $ [\psi^{-1}\left(\frac{|w|^2}{4}\right),\frac{R}{2}]$, we have $r>\psi^{-1}\left(\frac{|w|^2}{4}\right)$, and so since $\rho_\star\geq r$ we obtain
\begin{equation}\label{eq:star}
 \psi(\rho_\star)\le\psi\left(\psi^{-1}\left(\frac{|w|^2}{4}\right)\right)=\frac{|w|^2}{4},
\end{equation}
and plugging this into the equation for $\rho_\star$ we obtain
\begin{equation*}
 \frac{r}{\rho_\star}=1-\frac{2\psi(\rho_\star)}{|w|^2} \geq 1-\frac{2|w|^2}{4|w|^2}=\frac{1}{2}.
\end{equation*}

The monotonicity of $\psi$ ensures that, for the same relative velocity, on the interval $ \left[\frac{1}{2\sqrt 2}\psi^{-1}\left(\frac{|w|^2}{4}\right),\psi^{-1}\left(\frac{|w|^2}{4}\right)\right]$, the minimum radius for impact parameter $r$ is smaller than the minimum radius for $\psi^{-1}\left(\frac{|w|^2}{4}\right)$, and so using \eqref{eq:star} we obtain 
\begin{equation*}
 \rho_\star \leq 2 \psi^{-1}\left(\frac{|w|^2}{4}\right).
\end{equation*}
Since $r>\frac{1}{2\sqrt 2}\psi^{-1}\left(\frac{|w|^2}{4}\right)$ we have $\rho_\star \leq 4\sqrt{2}r$ as required.
\end{proof}

\subsection{Solutions of the Linear Boltzmann Equation}\label{sec:slns}

This specification of the two body interaction then allows us to comment on the types of solution of the linear Boltzmann equation one can obtain. The primary question is in which sense do the cancellation effects of the gain and loss parts of the collision operator manifest themselves. 

For an admissible long range potential $\phi$, we claim that this cancellation can only be considered in an integrated sense, as opposed to the total variation sense given in equation~\eqref{eq:strongcollop}.
Indeed, the issue with this formulation is that it requires the difference $f'g_\star'-fg_\star$ to compensate the unbounded integration over the impact parameters in $\cS$, so that the product has finite integral. To ensure that $L(f)\in L^1_{\mathrm{loc}}$, if 
\[
f'\,g'_\star-f\,g_\star\sim g_\star\, (f'-f)\sim C\,g_\star\,r^{-s}\,|v_\star-v|,
\] 
which would be natural to assume if the background was at equilibrium and $f$ was differentiable, then one formally obtains
\begin{equation*}
\begin{aligned}
  \int_{\R^3}\int_0^{2\pi}\int_0^\infty (f'\,g'_\star- f\,g_\star) \,r\,|v_\star-v|\dd r \dd \zeta \dd v_\star&=  C\int_{\R^3}\int_0^\infty g_\star \,r^{1-s}\,|v_\star-v|^2\dd r \dd v_\star\\
  &\leq \int_{\R^3}(|v|^2+|v_\star|^2)\,g_\star \dd v_\star \int_0^\infty r^{1-s}\dd r
\end{aligned}
\end{equation*}
and so the difference $f'-f$ must decay with $s>2$ for this to converge on compact sets of $v$. To deduce this decay one requires the product $f\,g$ to be Lipschitz, and so in particular for an arbitrary $f\in L^1$, strong solutions would not exist.

This thus results in us requiring only a weak formulation, since by testing with $C^\infty$ functions, we ensure that the difference $h'-h$ for the test functions does decay of the right order of magnitude due to Lemma~\ref{eq:angerr}.

This situation is in contrast to the evolution for the potential $\phi^{R}$. Here one can split the collision operator into $
 L^{R} = L^{R}_{+} - L^{{R}}_{-}$
for 
\begin{equation*}\label{eq:srcollop}
 \begin{cases}
  L^{R}_{+}(f) &=\int_{\R^3} \int_{B_R} f(v^{\prime,R})\, g(v_\star^{\prime,R})\, |v_\star-v| \dd S\dd v_\star\\[6pt]
  L^{{R}}_{-}(f)&= f(v)\int_{\R^3} \int_{B_R}  g(v_\star)\, |v_\star-v| \dd S\dd v_\star,
 \end{cases}
\end{equation*}
since due to the cut off, both are in $L^1(\cU)$ for any $f\in L^1(\cU, (1+|v|^2)\dd x \dd v)$. Furthermore, the norms of these operators depend on $R$ and in particular tend to infinity as $R\to \infty$.

One can then show that the operator $-v\cdot \nabla_x -L^{R}_{-}$ is a closed operator from its domain to $L^1$, and so by defining $\cT$ to be the semi-group generated by  $-v\cdot \nabla_x -L^{R}_{-}$, we can consider mild solutions of the linear Boltzmann equation. These are a function $f^R\colon [0,T]\times \cU \to \R$ such that for all $t\in [0,T]$ we have $f^R(t)\in L^1(\cU,(1+|v|^2)\dd x \dd v)$ and for which 
\begin{equation}\label{eq:mildsol}
 f^R(t,x,v) = \cT(t)f_0(x,v) +\int_0^t \cT(t-s) L^{R}_{+}(f^R)(s,x,v) \dd s.
\end{equation}
We remark that if such a mild solution exists and is unique, then by \cite{ball1977} it is also the unique weak solution in the sense of equation \eqref{eq:weaksol}. We now show that mild solutions exist.

\begin{proposition}\label{prop:SRexist}
For any $T>0$, there exists a unique mild solution \eqref{eq:mildsol} to the linear Boltzmann equation on $[0,T]$ with interaction potential $\phi^{R}$ such that 
 \begin{equation*}
  \int_\cU (1+|v|^2)\,f^{R}(t,x,v)\dd x \dd v<\infty
 \end{equation*}
for all $t\in [0,T]$.
\end{proposition}
\begin{proof}
We aim to apply \cite[Thm. 10.28]{banasiak2006} which will imply that $-v\cdot \nabla_x +L^{R}$ generates an honest semigroup, meaning that it is mass and energy preserving. This then implies existence and uniqueness due to \cite[Thm 3.1.12]{arendt2011}. The key conditions we must check are the following:
 \begin{enumerate}
\item Writing $L^R_-(f)= \nu(v)f(v)$, we require that for any $V>0$ there exists $M<\infty$ such that for $|v|\leq V$ we have $\nu(v)\leq M$.
\item The operator $L^R_+$ is an integral operator with kernel given by $k(v,v')$ with $k$ measurable and non-negative such that 
  \begin{equation*}
   \int_{\R^3} k(v',v)\dd v' =\nu(v)
  \end{equation*}
\item There exists $C>0$ such that for any fixed $V>0$ we have 
  \begin{equation*}
   \int_{|v'|>V} k(v',v) \dd v' \leq C
  \end{equation*}
  for almost all $|v|\leq V$. 
 \end{enumerate}

To satisfy (1), we can bound $\nu(v)$ easily, since for any $V$ with $|v|\leq V$ we have
 \begin{equation*}
|  \nu(v)|\leq C\, V\,R^2\norm{(1+V^2)\,g}_{L^1}.
 \end{equation*}
Using the Carleman representation, as in \cite{villani1999}, we can rewrite $ L^{R}_{+}(f)$ as
\begin{equation*}
  L^{R}_{+}(f) = \int_{\R^3}\frac{f(v')}{|v-v'|^2}\,\int_{E_{vv'}} g(v'_\star) \,b^R\left(\frac{|v'-v|}{|v'-v'_\star|},\,|v'-v'_\star|\right) \dd v'_\star \dd v'
\end{equation*}
where 
\begin{equation*}
 E_{vv'}= \{u\in \R^3 : u\cdot(v'-v)=v\cdot (v'-v)\}
\end{equation*}
and $b^R$ is the cross section for the potential $\phi^{R}$ in terms of the deviation angle $\theta^R$ and the relative velocity $v_\star-v$. Defining 
\begin{equation*} 
k(v,v')= \frac{1}{|v-v'|^2}\,\int_{E_{vv'}} g(v'_\star) \,b^R\left(\frac{|v'-v|}{|v'-v'_\star|},\,|v'-v'_\star|\right) \dd v'_\star 
\end{equation*}
we have the kernel of the gain part of the collision operator. Another simple coordinate change gives the second part of condition (2).

Finally, to show condition (3), we proceed as in \cite[Thm 2.1]{arlotti2007}, and rewrite $v'_\star=v'+V_2$ to obtain
\begin{equation*}
k(v,v') = \frac{1}{|v-v'|^2}\int\limits_{V_2\cdot(v'-v)=0}  g(v' +V_2) \,b^R\left(\frac{|v'-v|}{|v'-v-V_2|},\,|v'-v-V_2|\right) \dd V_2.
\end{equation*}
Transforming coordinates of $b^R$ from $r, |w|$ into $\theta^R, |w|$ we observe that
\begin{equation*}
 b^R(\theta,|w|) \leq C\, \sin \theta\, |w|
\end{equation*}
and therefore
\begin{equation*}
k(v,v') \leq \frac{C}{|v-v'|}\int\limits_{V_2\cdot(v'-v)=0}  g(v' +V_2) \dd V_2.
\end{equation*}
The same arguments as in \cite{Matthies2018}, where the calculations follow \cite[Ex 10.29]{banasiak2006},  with the conditions~\eqref{eq:bckgrnd}, show that,
\begin{equation*}
\int_{|v|>V}k(v,v')\dd v \leq C\int_0^\infty r \int_0^\infty z \sup_{|v|^2=|z|^2+|r|^2}g(v) \dd z \dd r<\infty
\end{equation*}
which thus gives condition (3). Arguments in \cite{Matthies2018} can be used to show the estimate
\begin{equation*} 
 \int_{\cU}(1+|v|^2) \, f^R(t,x,v) \dd x \dd v <\infty
\end{equation*}
which concludes the proof. 
\end{proof}

 \section{Marked Trees}\label{sec:MT}
  
In attempting to exploit the Markovian nature of the linear Boltzmann equation, we aim to create a Markovian description of the particle dynamics.  This cannot be performed on the space $\cU$, and so we enlarge the state space. The issue with a recollision free dynamic are the historic recollisions, and so we  introduce a space of marked trees with height 1 to encode the collisions of the tagged particle, with the aim of taking care of history dependence. The outcomes of the section are the following. Firstly we aim to obtain equations describing the evolution of the marginals $f^{\eps,R}$ and $P^R$ on a subspace of trees corresponding to recollision free dynamics. Secondly we aim to show that the set of these recollision free trees, or ``good trees,'' has almost full measure, and that the empirical density $f^{\eps,R}$ converges to the idealised density $P^R$ for the linear Boltzmann equation associated to $\phi^R$.

First, let $(x^{\eps,R},v^{\eps,R})$ be the solutions of the equations~\eqref{eq:partdens} with potential $\phi^R$. From this evolution, one can specify the geometric parameters $(r,\zeta)$ of each collision, as well as the velocities of the background particles.  One then encodes these parameters in the space of marked trees, which is the following.
  
\begin{definition}\label{def:MT}
 The set of marked trees $\cMT$ is defined by
 \begin{multline*}
  \cMT:=\Big\{(x_0,v_0),(t_1,r_1,\zeta_1,v_1),\dots,(t_n,r_n,\zeta_n,v_n)\,|\, (x_0,v_0)\in \cU,\, t_i\in [0,T],\\r_i\in[0,R],\, \zeta_i\in [0,2\pi),\,v_i \in \R^3,\, n\in \N\cup\{0\} \Big\}
 \end{multline*}
and we furthermore define, for a tree $\Phi \in \cMT$, the function $n(\Phi)=n$ to be the number of collisions. We set $\cMT_k=\{\Phi\in \cMT:n(\Phi)=k\}$. Also, defining 
\begin{equation*}
\tau(\Phi):=\begin{cases}0 & n(\Phi)=0\\\max_{1\leq j \leq n }t_j& \mathrm{else},\end{cases}
\end{equation*}
we typically denote the final marker by 
\begin{equation*}
 (\tau, \bar r,\bar \zeta,\bar v):=(t_n,r_n,\zeta_n,v_n).
\end{equation*}
\end{definition}
We now describe the density $f^{\eps,R}$ on $\cMT$ derived from the particle dynamics with potential $\phi^R$. Associated to each tree $\Phi \in \cMT$ are dynamics $(x^{\eps,R},v^{\eps,R})$ which correspond to evolution starting at $(x_0,v_0)$ and encountering a collision with a background particle corresponding to each node of the marked tree. These evolutions can be inconsistent with the dynamics given by solutions of \eqref{eq:partdens} in two ways. Firstly, these dynamics could be unphysical, i.e. they miss collisions, and secondly, the dynamics described could encompass recollisions. Without the removal of recollisions, we are not able to uniquely specify from the tree the dynamics of the background particles in each collision. 

The first error is nullified by the density $f^{\eps,R}$ not being supported on such trees. The second is removed by restricting the dynamics onto a space $\cG(\eps)$ of trees that is given by a combination of  Definition~\ref{def:goodtrees} below and Definition~\ref{geps}, so that, among other restrictions, recollisions are not observed. The latter definition contains the various parameters required to show convergence of $f^{\eps,R}$ to $P^R$, as these are not needed for the specification of the particle density.

\begin{definition}\label{def:goodtrees}
 Trees $\Phi\in\cMT$ exhibit {\bf good dynamics} at spatial scale $\eps>0$ if they satisfy the following properties.
 \begin{enumerate}
\item The velocities have a minimum separation, meaning that $ \min_{i=1,\dots,n(\Phi)} |v^\eps(t_i^-)-v_i|>0$.
\item The times $t_i$ are such that for all $i=2,\dots , n(\Phi)$ we have $t_i-t_{i-1}>0$.
\item There is no initial overlap at diameter $\eps$ meaning that for all $j=1,\dots,N$ we have $ |x_0-x_j(0)|>R\eps$.
\item The trees are recollision free at diameter $\eps$ meaning, for all $0\leq \eps'\leq R\eps$, for all $1\leq j \leq n(\Phi)$ and for all $t\in[0,T]\setminus (t_j,t_j+\tau_\star^j)$, one has
 \begin{equation*}
  |x^\eps(t)-(x_j+tv_j)|>\eps'.
 \end{equation*}
 \end{enumerate}
\end{definition}
In order to compare the particle density with the corresponding density for the Boltzmann evolution, we write an evolution equation for the particle density on the space $\cMT$. This density evolves from a tree $\Phi$ by addition of a node which represents a collision. To describe this, we denote $\overline \Phi$ as the tree $\Phi$ with the final collision removed.  This evolution equation is in general a complicated expression in the form
\begin{multline}\label{eq:partdensmt}
  -\int_0^T\int_{\cMT}f^{\eps,R}_t (\Phi) \, \pd_t h_t(\Phi) \dd \Phi \dd t- \int_{\cMT_0} \xi \,f_0(x_0,v_0)\, h_0(\Phi)\dd \Phi\\= \int_{\cMT}\cQ^{\eps,+}(\Phi)\,f^{\eps,R}_\tau(\overline\Phi) \,h_\tau(\Phi)\dd \Phi-\int_0^T \int_{\cMT} \cQ^{\eps,-}(\Phi)\,f_t^{\eps,R}(\Phi)\, h_t(\Phi)\dd \Phi \dd t 
\end{multline}
 for $h$ a test function, where the terms on the right hand side describe the effect of a collision on the particle density. In particular, $\cQ^{\eps,+}$ describes the rate at which one encounters a background particle with final collision given by the parameters in $\Phi$, and so the gain of density onto $\Phi$ from a collision. $\cQ^{\eps,-}$ describes the probability of the tagged particle colliding with another background particle, and so the loss of density from $\Phi$. The tree based formulation should be thought of akin to the Lagrangian formulation of the evolution equation, and so one does not have a spatial derivative.

While the coefficients $\cQ^{\eps,+}$ and $\cQ^{\eps,-}$ describe correlations between all collisions, these build up slowly, and for trees exhibiting good dynamics, the dependency upon the tree itself is sufficiently weak to make it worth stating the limiting equation for the Boltzmann dynamics. 

Analogously to $(x^{\eps,R},v^{\eps,R})$, we define
\begin{equation*}
 (x^R,v^R)\colon [0,T)\times \cMT \to \cU
\end{equation*}
by 
\begin{equation*}
 \begin{cases}
  v^R(t) = v_0 & t\in [0,t_1)\\
  v^R(t) = \sigma^R_1(r_i,\zeta_i,v^R(t_{i-1}),v_i) &t \in [t_i,t_{i+1})\\
  x^R(t) = x_0 + \int_0^t v^R(s) \dd s& 
 \end{cases}
\end{equation*}
to be the corresponding Boltzmann dynamics on $\cMT$, where one notes that $\sigma^R_1$ is the first component of the scattering map as defined in \eqref{eq:scatop}.  Defining, for $h$ a suitable test function, the relationship
\begin{equation*}
 \int_{\cU} f^R(t,x,v)\, h(x,v) \dd x \dd v =: \int_\cMT P_t^R(\Phi)\, h(x^R,v^R) \dd \Phi
\end{equation*}
and then inferring from the Dyson-Philips formula a notion of collision in $f^R$, we can show existence of such a function, and then that it satisfies the evolution equation on $\cMT$ of 
\begin{multline}\label{eq:lbedensmt}
  -\int_0^T\int_{\cMT}P^{R}_t (\Phi) \, \pd_t h_t(\Phi) \dd \Phi \dd t- \int_{\cMT_0} f_0(x_0,v_0)\, h_0(\Phi)\dd \Phi\\= \int_{\cMT}\cQ^{+}(\Phi)\,P^{R}_\tau(\overline\Phi) \,h_\tau(\Phi)\dd \Phi-\int_0^T \int_{\cMT} \cQ^{-}(\Phi)\,P_t^{R}(\Phi)\, h_t(\Phi)\dd \Phi \dd t, 
\end{multline}
where one should recall that $\overline \Phi$ is the tree $\Phi$ with final collision removed.

The similarity in the forms of \eqref{eq:partdensmt} and \eqref{eq:lbedensmt} then enables the easy comparison of the particle and Boltzmann densities on $\cMT$. We first find the explicit forms of the parameters in these equations.
 
\subsection{Short Range Particle Dynamics on Marked Trees}

We now derive an effective evolution equation for the particle dynamics on $\cMT$. While we derive a strong form of the equation, we remark that the equation takes the weak form in equation \eqref{eq:partdensmt}. The main aim for this subsection is to calculate the coefficients $\cQ^{\eps,+}$ and $\cQ^{\eps,-}$, at least for the trees that satisfy Definition~\ref{def:goodtrees}.

To ensure that the particle evolution describes all collisions, we first must ensure that there are no collisions initially, and so remove those background particles that initially collide. To ensure that $f_0^{\eps,R}$ is a probability measure, we multiply $f_0$ by 
\begin{equation*}
 \xi(\eps,R)= \left(1-\frac{4}{3}\,\pi\, R^3\,\eps^3\right)^N.
\end{equation*}
The jump rate $\cQ^{\eps,+}$  is calculated by describing the probability of finding a background particle at the correct place to collide with the tagged particle at time $t$. We must normalise by ensuring that the background particle has never collided at an historic time with the tagged particle. Therefore, recalling the notation $\bar r$ in Definition~\ref{def:MT}, the creation of density on $\Phi$ occurs with probability 
\begin{equation*}
    \cQ^{\eps,+}[f^{\eps,R}_t](\Phi)=\begin{cases}
(1-\gamma(\eps,\tau))\1_{t=\tau(\Phi)} \,\1_{t-\tau(\overline \Phi)>\delta}\,\frac{ \bar r\, |\bar v-v^{\eps,R}(t)|\, g(\bar v)\, f^{\eps,R}_t(\overline \Phi)}{1-\eta_\tau^{\eps,R}(\Phi)} &n(\Phi)>0\\
0&n(\Phi)=0
 \end{cases}
\end{equation*}
where the term
\begin{equation*}
 \eta_t^{\eps,R} (\Phi)= \int_\cU g(v_\star)\, \left(1-\1_t^{\eps,R}[\Phi](x_\star,v_\star)\right)\dd x_\star \dd v_\star
\end{equation*}
describes the normalisation factor ensuring that the tagged particle is not recolliding with the background, and where
\begin{equation*}
 \1_t^{\eps,R}[\Phi](x_\star,v_\star)= \begin{cases}1& \text{if for all } s\in(0,t) \text{ we have } |x^{\eps,R}_\Phi(s)-(x_\star+sv_\star)|>\eps\\0&\text{else}\end{cases}
\end{equation*}
is the indicator function for ensuring there are no historic recollisions with tagged particle starting at $(x_\star,v_\star)$, and where $\gamma(\eps,\tau)=n(\overline\Phi)\eps^2$.

The loss term for density on marked tree $\Phi$ is given similarly, except one is required to calculate the probability of encountering any background particle that has not yet collided. This means that the loss term is given by
\begin{equation*}
  \cQ^{\eps,-}(\Phi)=(1-\gamma(\eps,t))\1_{t-\tau(\Phi)>\delta}\,\frac{\int_{\R^3} \int_{B_R} g(v_\star)\,|v^{\eps,R}(t)-v_\star|\dd S\dd v_\star -c(\eps,R)}{1-\eta_t^{\eps,R}(\Phi)}
\end{equation*} 
where $c(\eps,R)$ is a term that quantifies the probability that the background particle has already collided with the tagged particle, and $\gamma(\eps,t)=n(\Phi)\eps^2$.

The main addition to this analysis from considering hard sphere dynamics in \cite{Matthies2018} are the extra restrictions on the dynamics to ensure each collision includes only two particles, which is the use of the term $\delta$. The restrictions in Definition~\ref{def:goodtrees} then enable the relatively simple specification of the evolution of the particle density. We have the following strong form for the evolution.

\begin{lemma}\label{lem:partontree}
 The tagged particle density function on a tree $\Phi$ that satisfies Definition~\ref{def:goodtrees} evolves via the equation 
 \begin{equation}\label{eq:srpart}
 \begin{cases}
  \pd_t f^{\eps,R}_t(\Phi)= \cQ^{\eps,+}[f^{\eps,R}_t](\Phi)-f^{\eps,R}_t(\Phi)\,\cQ^{\eps,-}(\Phi)\\
  f^{\eps,R}_0(\Phi) = \xi(\eps,R)\1_{\cMT_0}(\Phi)f_0(x_0(\Phi),v_0(\Phi)).
 \end{cases}
\end{equation}
\end{lemma}
The proof is a careful calculation of the relevant probabilities and is \cite{Matthies2018} mutatis mutandis.

We remark that this evolution equation provides an $L^\infty$ estimate on $f^{\eps,R}$ of
\begin{equation*}
 f_t^{\eps,R}(\Phi) \leq (4RV_2(\eps))^{n(\Phi)}\leq (4RV_2(\eps))^{M(\eps)}.
\end{equation*}

The crucial observation on the coefficients in equation \eqref{eq:srpart} is that they only depend on the tree $\Phi$ based upon its implicit particle evolution $(x^{\eps,R},v^{\eps,R})$ and its parameters of the final collision, and so in this manner the dependency upon the tree itself is weak. 

 \subsection{Boltzmann Equation on Marked Trees}
 
 We now prove that the probability density on $\cMT$ exists, and satisfies an equation of the form \eqref{eq:lbedensmt}. The main aim of this subsection is to show that we have an idealised equation on the space $\cMT$, and that a solution of this equation can be related to a mild solution of the linear Boltzmann equation in the form \eqref{eq:mildsol}.  
 
 The gain in density on tree $\Phi$ is given by the probability of encountering a background particle that collides with the tagged particle, so has intensity given by
 \begin{equation*}
  \cQ^+[P^R_t](\Phi)=\begin{cases}
                       \1_{t=\tau(\Phi)}\,P^R_t(\overline \Phi)\,g(\bar v)\,\bar r\,|v^{\eps,R}(\tau^-)-\bar v| & n(\Phi)>0\\
                       0&n(\Phi)=0,
                      \end{cases}
\end{equation*}
and the loss of density is given by the probability of encountering any background particle, so is given by 
\begin{equation*}
   \cQ^-_t(\Phi)=\int_{\R^3} \int_{B_R} g(v_\star)\,|v^{\eps,R}(t)-v_\star| \dd S\dd v_\star.
\end{equation*}
These observations then give the following lemma for the evolution equation corresponding to the linear Boltzmann equation. 
 \begin{lemma}\label{lem:BEontree}
 There exists a solution $P^R\colon [0,T]\to L^1(\cMT)$ to the equation
 \begin{equation}\label{eq:srBE}
  \begin{cases}
   \pd_t P^R_t(\Phi) =\cQ^+[P^R_t](\Phi)-P^R_t(\Phi)\,\cQ^-_\tau(\Phi)\\
   P^R_0(\Phi) = f_0(x_0,v_0 )\,\1_{\cMT_0}(\Phi)
  \end{cases}
 \end{equation}
for $f_0\in L^1(\cU, (1+|v|^2)\dd x \dd v)$ and $f_0\geq 0$. Then defining, for $\Omega \subset \cU$,
\begin{equation*}
 S_t(\Omega)=\{\Phi \in \cMT : (x^R(t),v^R(t))\in \Omega\},
\end{equation*}
we have
\begin{equation}\label{eq:treeeqcomp}
 \int_\Omega f^R(t,x,v)\dd x \dd v= \int_{S_t(\Omega)}P^R_t(\Phi)\dd \Phi
 \end{equation}
for $f^R$ the unique mild solution of the linear Boltzmann equation for $\phi^R$, as in equation~\eqref{eq:mildsol}. 
\end{lemma}

In comparison with \eqref{eq:srpart}, the coefficients in the formula \eqref{eq:srBE} are simple, and yet of the same form. Furthermore, they describe a dependence upon the tree $\Phi$ that can be seen in the proof to be consistent with the linear Boltzmann equation.

\begin{proof}
The proof proceeds by constructing explicitly such a function $P^R_t$. We give the barest details as the proof proceeds as in \cite{Matthies2018} mutatis mutandis.

 Using \cite[Thm 10.4]{banasiak2006}, the Hille-Yosida theorem and \cite[Thm 3.1.12]{arendt2011}, there exists a unique mild solution to 
 \begin{equation*}
  \begin{cases}
   \pd_t P_t^{(0)}(x,v)= -v\cdot \nabla_x P_t^{(0)}(x,v) -L^{R}_{-}(P_t^{(0)})(x,v)\\
   P_0^{(0)}(x,v)=f_0(x,v).
  \end{cases}
 \end{equation*}
We then define 
\begin{equation*}
 \begin{cases}
  P^R_t(\Phi)=P_t^{(0)}(x^R(t),v^R(t)) &  \Phi \in \cMT_0\\[4pt]
  P^R_t(\Phi)= \1_{t\ge \tau}\, e^{-(t-\tau)\,\cQ_\tau^-(\Phi)}\,P^R_\tau(\overline\Phi)\,g(\bar v)\,\bar r\,|v^{\eps,R}(\tau^-)-\bar v|& \mathrm{otherwise},
 \end{cases}
\end{equation*}
and we have a unique solution to this equation. Defining 
\begin{equation*}
 P_t^{(j)}(S)= \int_{S\cap\cMT_j}P^R_t(\Phi)\dd \Phi
\end{equation*}
it is an easy application of \cite[Prop 3.31]{arendt2011} to show that $\sum_{j=0}^\infty P_t^{(j)}$ is a mild solution of the linear Boltzmann equation \eqref{eq:mildsol}. By uniqueness of solutions we have $\sum_{j=0}^\infty P_t^{(j)}=f^R$, from which one can deduce that $P^R\in L^1(\cMT)$ and equation~\eqref{eq:treeeqcomp}. 
\end{proof}

\subsection{Markovian Convergence}

We have created probability densities $f^{\eps,R}$ and $P^R$ corresponding to short range particle dynamics and the linear Boltzmann equation associated to $\phi^R$. Furthermore, for a subset of $\cMT$ both densities admit a Markovian evolution equation. We now aim to exploit this Markovian evolution to show convergence of the density $f^{\eps,R}$ to $P^R$ in the limit $\eps\to 0$.

We first state the precise orders of magnitude of the restrictions of good trees that we use.

\begin{definition}[Good Trees]\label{geps}
 Let
 \begin{equation*}
  M,V_2\colon (0,1) \to \R_+, \hspace{1cm} M(\eps)=|\log\eps|=V_2(\eps)
 \end{equation*}
 and 
 \begin{equation*}
  V_1,\delta\colon (0,1)\to \R_+, \hspace{1cm} V_1(\eps)=\frac{1}{|\log\eps|},\,\,\, \delta(\eps)=\sqrt\eps.
 \end{equation*}
The set $\cG(\eps)$ of good trees is then the set of trees $\Phi \in \cMT$ that exhibit good dynamics (satisfy Definition~\ref{def:goodtrees}), and that satisfy the following:
 \begin{enumerate}
  \item $
   \max\left\{\sup\limits_{t\in[0,\tau]}|v^{\eps,R}(t)|,\max |v_j|\right\}\leq V_2(\eps)$.
\item For all $i=1,\dots,n(\Phi)$, we have $
 \min\limits_{i=1,\dots,n(\Phi)} |v^{\eps,R}(t_i^-)-v_i|\geq V_1(\eps)$.
\item For all $\Phi \in \cG(\eps)$, we have $n(\Phi)\leq M(\eps)$,
\item For all $i=2,\dots,n(\Phi)$ we have $
 |t_i-t_{i-1}|>\delta(\eps)$.
 \end{enumerate}
\end{definition}

This choice of parameters should be thought of as a technical tool that enables us to deduce convergence of $f^{\eps,R}$ to $P^R$. The important requirements that are needed are that $ V_1(\eps),\,\, \delta(\eps) \to 0$ and $ V_2(\eps),\,\,M(\eps)\to \infty$, as well as ensuring that $R\eps \ll \delta V_1$ which comes from the bound on the scattering time in Lemma~\ref{lem:angdev}. From these, the specific form of the parameters chosen are used so that the estimates of the difference between the jump intensities and decay rates for $f^{\eps,R}$ and $P^R$ tend to zero as $\eps \to 0$.

\begin{lemma}\label{srdensity}
 Suppose that $f^{\eps,R}$ and $P^R$ are the probability densities on $\cMT$ corresponding to the equations \eqref{eq:partdens} for short range potential $\phi^{R}$, and for the Boltzmann equation \eqref{eq:weaksol} associated to $\phi^{R}$. We then have for all $t\in [0,T]$ that
 \begin{equation*}
  \norm{f_t^{\eps,R}-P^R_t}_{TV} \to 0
 \end{equation*}
as $\eps \to 0$ with $N\eps^{2}=1$.
\end{lemma}

\begin{proof}
The proof proceeds as in \cite{Matthies2018}, whereby we claim that if we have an inequality, for $\Phi \in \cG(\eps)$, of the form 
 \begin{equation}\label{eq:ineq}
  f_t^{\eps,R}(\Phi)-\xi(\eps,R)\,P^R_t(\Phi)\geq -\rho^\eps_t(\Phi)\,\xi(\eps,R)\,P^R_t(\Phi),
 \end{equation}
for some $\rho^\eps_t(\Phi)$ uniformly bounded in $\Phi$ with this bound decaying to $0$ as $\eps\to 0$, together with 
\begin{equation*}
 P^R_t(\cMT\setminus \cG(\eps))\to 0,
\end{equation*}
then we can conclude. 

Indeed, if the inequality \eqref{eq:ineq} holds, then we can deduce that
\begin{equation*}
 \int_{S\cap \cG(\eps)} \left(f_t^{\eps,R}(\Phi)-\xi(\eps,R)\,P^R_t(\Phi)\right)\dd \Phi\geq -\xi(\eps,R)\sup_{\Phi \in \cMT}\rho^\eps_t(\Phi)
\end{equation*}
where we have bounded $P^R_t(S)\leq 1$.
We then have, for any $S\subset \cMT$, the inequality
 \begin{equation*}
  \begin{aligned}
   P^R_t(S)- f_t^{\eps,R}(S)&\leq P^R_t(S\cap \cG(\eps))+P^R_t(S\setminus \cG(\eps))- f_t^{\eps,R}(S\cap G(\eps))\\ 
   &= P^R_t(S\cap \cG(\eps))-\xi(\eps,R)\,P^R_t(S\cap \cG(\eps))\\&\hspace{1cm}-\left( f_t^{\eps,R}(S\cap \cG(\eps))-\xi(\eps,R)\,P^R_t(S\cap \cG(\eps))\right) +P^R_t(S\setminus \cG(\eps))\\
   &\leq (1-\xi(\eps,R))\,P^R_t(S\cap \cG(\eps))+\xi(\eps,R)\,\sup_{\Phi \in \cMT}\rho^\eps_t(\Phi)+P^R_t(S\setminus \cG(\eps))\\
  \end{aligned}
 \end{equation*}
 and then the final two terms tend to zero by assumption. By analysing the form of $\xi$, one observes that this tends to $1$ as $\eps\to 0$. 
 
 We are thus left with justifying the two assumptions we made. To address the first, we define, for $\mu =e^{\delta \,\sup_t \cQ_t^-(\Phi)}$ and for $k=1,\dots,n(\Phi)$, the quantities
  \begin{equation}\label{eq:rhoterm}
  \begin{aligned}
     \rho^{\eps,0}_t(\Phi):&= 2\,t\,\eta_t^{\eps,R}(\Phi)\, \sup_{t\in[0,T]} \cQ^-_t(\Phi)\,\left(1+\delta\,\sup_{t\in[0,T]} \cQ^-_t(\Phi)\right),\\
       \rho_t^{\eps,k}(\Phi)&= \mu \left(\eps+(1-\eps)\rho_t^{\eps,k-1}(\Phi)\right) + \rho_t^{\eps,0}(\Phi)
  \end{aligned}
\end{equation}
and then set $\rho_t^\eps(\Phi)=\rho_t^{\eps,n(\Phi)}(\Phi)$.
We now analyse the decay of this as $\eps\to 0$. We first note that 
\begin{equation*}
  \sup_{t\in[0,T]} \cQ^-_t(\Phi)\leq C\,R^2\,\left(\int (1+|v_\star|^2)\,g_\star\dd v_\star+V_2(\eps)\right)
\end{equation*}
via an easy application of the triangle inequality. Furthermore, by analysing the area of space for which the indicator function $\1_t^\eps[\Phi]=0$, we can conclude that
\begin{equation*}
 \eta_t^{\eps,R}(\Phi) \leq C\, R^2 \,\eps^2 \,T \,\left(\int (1+|v_\star|^2)\,g_\star\dd v_\star+V_2(\eps)\right).
\end{equation*}

The decay of $M(\eps),\,V_2(\eps)$ and $ \delta$ in Definition~\ref{geps} then ensures that all terms in $\rho_t^\eps(\Phi)$ converge to $0$ as $\eps\to 0$ as required.

We now show the inequality \eqref{eq:ineq}. The aim is to prove this by induction on the number of nodes of $\Phi$. We show the inequality using the following three steps;
\begin{enumerate}
 \item First we show that one has an estimate for the deviation of $f_t^{\eps,R}-\xi(\eps,R)P^R_t(\Phi)$ at time $t>\tau(\Phi)$, given the difference at time $\tau(\Phi)$, together with $\eta^{\eps,R}_t$ and $\rho_t^{\eps,0}$.
 \item Secondly we show that $1-\frac{1-\gamma(\eps,t)}{1-\eta_\tau^{\eps,R}(\Phi)}\leq \eps$.
 \item Finally the above two steps are combined in an iterative argument, which gives the precise form  of $\rho_t^\eps(\Phi)$ as in equation \eqref{eq:rhoterm}.
\end{enumerate}

The second two steps are the same as in \cite{Matthies2018} and so are omitted. We briefly elucidate step (1) as the inclusion of the time separation $\delta$ makes the analysis different. 

We first split time period into $[t,t+\delta]$ and $[t+\delta,T]$. Comparing the equations for $P^R_t$ and $f_t^{\eps,R}$ one can say that, for $t\in [\tau,\tau+\delta]$ we have
\begin{equation*}
 f_t^{\eps,R}(\Phi)-\xi(\eps,R)\,P^R_t(\Phi) = f_\tau^{\eps,R}(\Phi)-\xi(\eps,R)\,P^R_\tau(\Phi) +\xi(\eps,R)\, P^R_t(\Phi) \left(e^{(t-\tau)\,\cQ_\tau^-(\Phi)}-1\right)
\end{equation*}
by direct integration. Since the final term is positive, we can easily bound it from below by $-\rho_t^{\eps,0}(\Phi)\, P^R_t(\Phi)$. For $t>\tau+\delta $ the method is similar in nature to $t<\tau+\delta$, and is exactly as in \cite{Matthies2018}.

Finally, to show $P^R_t(\cMT\setminus \cG(\eps))\to 0$ we note that the constraints of $\cG(\eps)$, and the order of magnitude of the constraints $M,\,V_1,\,V_2$ and $\delta$  restrict onto a set of vanishing measure in the limit $\eps\to 0$. The proof follows \cite[Prop. 7]{Matthies2018}
mutatis mutandis.
\end{proof}
 
 \section{Short and Long Range Particle Dynamics}\label{sec:partdyn}
 
 As we have compared short range and linear Boltzmann densities in the previous section, to prove Theorem~\ref{thm:main} we must provide an analysis of how the inclusion of grazing collisions in the dynamics alters these densities. We first show how the use of the potential $\phi$ changes the particle dynamics in comparison with dynamics with potential $\phi^R$.  The main aim of this section is to show that $f^{\eps,R}-f^\eps$ converges weakly to $0$ as $\eps\to 0$ with $R$ a function of $\eps$.
 
 Since the long range evolution is not Markovian, the evolution equation derived in the previous section for $f^{\eps,R}$ is not useful, other than in providing $L^\infty$ estimates. Furthermore, since evolution under potential $\phi$ has no well defined notion of collision, relating long range dynamics in a deterministic relationship with a tree $\Phi \in \cMT$ is meaningless, and so this section contains a major difference to \cite{Matthies2018}. The issue is that the background particles not described by $\Phi$ still interact with the tagged particle in the setting of long range dynamics. 
 To take care of this, for each tree $\Phi\in \cG(\eps)$, we introduce random variables $(x^\eps,v^\eps)$ that describe the long range evolution under the $n(\Phi)$ background scatterers in $\Phi$, where we assume the remaining $N-n(\Phi)$ background are randomly placed so that they do not interact with the short range dynamics $x^{\eps,R}$ described in $\Phi$.
 
 \subsection{Long and Short Range Evolutions with the same near collisions}\label{sec:nearcoll}
 
 The first estimate we obtain is an estimate on the maximum error between the short range evolution $(x^{\eps,R},v^{\eps,R})$ and the random variables $(x^\eps,v^\eps)$ corresponding to the long range evolution, under the assumption that they both encounter the same background particles in near collisions. We use the form of the equations~\eqref{eq:partdens} to compare the solutions with differing interaction potentials as follows.
  
 \begin{lemma}\label{lem:SLdynerr}
 Let $\phi$ be an admissible potential with decay as in \eqref{eq:potdecay}, and let $k\in \N$ with $k\leq M(\eps)$, and let $R= \eps^{-1/(3+\gamma)}$.
 Let $(x^{\eps,R},v^{\eps,R})$ be the evolution for tree $\Phi\in \cG(\eps)\cap\cMT_k$, and let $(x^\eps,v^\eps)$ solve, for $t\in [0,T]$, the system
 \begin{equation*}
 \begin{cases} \dot x =v\\ \dot v= -\frac{1}{\eps}\sum_{i=1}^N \nabla \phi(\frac{x -x_i}{\eps})\end{cases}
\end{equation*}
with the same initial conditions and background as in $\Phi$, and assume that the remaining  $N-k$ background particles are distributed such that for all $t\in[0,T]$ we have
 \begin{equation*}
|x^\eps(t)-x_i|>R\eps,\,\,|x^{\eps,R}(t)-x_i|>R\eps.
 \end{equation*}
 Furthermore, suppose that there are times such that $|x^\eps(\cdot)-x_i(\cdot)|\leq R\eps$. 
 Then there exists $C>0$ depending on $\phi$ and $T$ such that, for all $t\in [0,T]$ we have,
 \begin{equation*}
  |x^{\eps,R}(t)-x^\eps(t)|+|v^{\eps,R}(t)-v^\eps(t)|\leq b(\eps)
\end{equation*}
 where 
 \begin{equation}\label{eq:beps}
 b(\eps) = Ce^{-C(1/\eps)^{\gamma/(3+\gamma)}}
\end{equation}
where $\gamma$ comes from the exponent in the decay \eqref{eq:potdecay} of the potential $\phi$.
\end{lemma}

We remark that we assume here that both $x^{\eps,R}$ and $x^\eps$ do not interact with the remaining $N-k$ background, and the assumption for $x^\eps$ is not valid for all background particle configurations. Furthermore, we note that the use of the removal of recollisions in Definition~\ref{def:goodtrees} of good dynamics ensures that the short range dynamics have exactly $k$ collisions with the background, and this is used in the proof.

\begin{proof}
Let $(\bar x^\eps,\bar v^\eps)$ be the long range evolution under the background in $\Phi$. The difference between $(x^{\eps,R},v^{\eps,R})$ and $(\bar x^\eps,\bar v^\eps)$ can be estimated by following the proof of \cite[Lemma 2]{Ayi2017}, where one combines the bound on $\tau_\star$ in Lemma~\ref{lem:angdev} with standard Gronwall estimates, which gives 
\begin{equation}\label{eq:ayiest}
|x^{\eps,R}(t)-\bar x^\eps(t)|+|v^{\eps,R}(t)-\bar v^\eps(t)|\leq Ck\frac{e^{CR\,V_1(\eps)^{-1}k}}{\eps^k}\norm{(1-\Lambda^R)\nabla \phi}_\infty.
\end{equation}

One then must estimate the deviation from this intermediary evolution when one involves all $N$ long range scatterers. By letting $z=\bar x^\eps-x^\eps$ and $w=\bar v^\eps-v^\eps$ we have that $(z,w)$ solves
 \begin{equation*}
   \begin{cases}
   \dot z=w\\
   \dot w= \frac{1}{\eps} \sum_{i=1}^k \left(\nabla \phi \left(\frac{x^\eps(t)-x_i}{\eps}\right)-\nabla \phi \left(\frac{\bar x^\eps(t)-x_i}{\eps}\right)\right)+ \frac{1}{\eps}\sum_{i=k+1}^N \nabla \phi \left(\frac{x^\eps(t)-x_i}{\eps}\right)\\
   z(0)=0\\
   w(0)=0.
  \end{cases}
 \end{equation*}
Using the Lipschitz nature of $\nabla\phi$ and the fact that $|x^\eps(t)-x_i|>R\eps$ results in 
 \begin{equation*}
   \begin{cases}
   \frac{\dd}{\dd t}| z|_1=|w|_1\\
   \frac{\dd}{\dd t}|w|_1\leq \frac{1}{\eps} k \frac{|z|_1}{\eps}+ \frac{N-k}{\eps}\norm{(1-\1_{B_R})\nabla \phi}_{L^\infty}.
  \end{cases}
 \end{equation*}
 Separating the variables, and then using the variation of constants formula enables one to write
  \begin{equation*}
  \begin{aligned}
   |z|_1&\leq \frac{(N-k)\norm{(1-\1_{B_R})\nabla \phi}_{L^\infty}}{2\eps\sqrt{Ck/\eps}}\left(\int_0^t e^{ \sqrt{Ck/\eps}(t-s)}- e^{- \sqrt{Ck/\eps}(t-s)}\dd s\right)\\
   |w|_1&\leq \frac{(N-k)}{2\eps}\norm{(1-\1_{B_R})\nabla \phi}_{L^\infty}\left(\int_0^t e^{ \sqrt{Ck/\eps}(t-s)}- e^{- \sqrt{Ck/\eps}(t-s)}\dd s\right)
  \end{aligned}
 \end{equation*}
 and by simplifying, we obtain
 \begin{equation}\label{eq:eggest}
   |\bar x^\eps(t)-x^\eps(t)|+|\bar v^\eps(t)-v^\eps(t)|\leq C\frac{e^{C\sqrt{k/\eps}}N}{ \sqrt{\eps}}\norm{(1-\Lambda^R)\nabla \phi}_{L^\infty}
 \end{equation}
 as required. Combining equations \eqref{eq:ayiest} and \eqref{eq:eggest}  and inputting the asymptotics of the parameters gives
 \begin{equation*}
  |x^{\eps,R}(t)-x^\eps(t)|+|v^{\eps,R}(t)-v^\eps(t)|\leq C\left(|\log\eps|\frac{e^{CR\,|\log \eps|^2}}{\eps^{|\log \eps|}} + \frac{e^{C\sqrt{|\log \eps|/\eps}}}{ \eps^{5/2}}\right) e^{-CR^{3/2+\gamma}} .
\end{equation*}
from which one can easily see that the right hand side is smaller than $b(\eps)$.
\end{proof}

\subsection{Size of Background Leading to Differing Collisional Structures}\label{sec:diffnearcoll}

The preceding section made certain assumptions on the background particles, so that their distribution ensured that, for tree $\Phi \in \cMT$, the random long range evolution encountered the same near collisions as the short range evolution. The purpose of this section is to describe the subset of $\cG(\eps)$ for which these conditions hold with high probability, as well as showing that this subset has probability $1$ in the limit $\eps \to 0$.

We must ensure two events happen. Firstly we must ensure that the long range evolution encounters the same near collisions, and secondly that it encounters the same grazing collisions. For the former, we define the following set.

 \begin{definition}\label{def:nongraz}
  We define the set $\cR(\eps)$ to be those trees $\Phi \in \cG(\eps)$ such that all impact parameter node labels are bounded by
  \begin{equation*}
   0\leq r_i\leq R-\frac{b(\eps)}{\eps}\left(1+\frac{1}{V_1(\eps)}\right)
  \end{equation*}
  where $b(\eps)$ is defined in equation~\eqref{eq:beps}, and $V_1$ in Definition~\ref{geps}.
 \end{definition}
 The purpose and form of this restriction is to ensure that the minimum radius of the two body interaction is smaller than  $R\eps-b(\eps)$, as we now show. This will ensure that the long range evolution does collide with the same background particles.
 
\begin{lemma}\label{lem:minrad}
 Suppose that for spatial scale $\eps>0$ in a binary collision under potential $\phi^R$, the impact parameter $r$ and relative velocity $w$ are bounded by
 \begin{equation*}
0\leq r \leq R-\frac{b(\eps)}{\eps}-\frac{b(\eps)}{\eps|w|},\,\,\,\, |w|\geq \frac{1}{|\log\eps|}.
 \end{equation*}
 Then for $\eps$ sufficiently small the minimum radius is bounded by
 \begin{equation*}
  \rho_\star \leq R\eps-b(\eps).
 \end{equation*}
\end{lemma} 
 \begin{proof}
  The minimum radius satisfies the equation
  \begin{equation*}
   1=\frac{r^2}{\rho^2_\star}+\frac{\frac{1}{\eps}\phi^R\left(\frac{\rho_\star}{\eps}\right)}{|w|^2}
  \end{equation*}
from conservation of angular momentum. Rearranging this, we obtain
  \begin{equation*}
   \rho_\star^2=r^2+\frac{\rho_\star^2\frac{1}{\eps}\phi^R\left(\frac{\rho_\star}{\eps}\right)}{|w|^2},
  \end{equation*}
  and inputting the constraint on $r$ into this equation results in
  \begin{equation*}
   \rho_\star^2=r^2+\frac{\rho_\star^2\frac{1}{\eps}\phi^R\left(\frac{\rho_\star}{\eps}\right)}{|w|^2}\leq (R\eps-b)^2 -\frac{2b}{|w|}(R\eps-b) +\frac{b^2}{|w|^2} +\frac{\rho_\star^2\frac{1}{\eps}\phi^R\left(\frac{\rho_\star}{\eps}\right)}{|w|^2},
  \end{equation*}
and to conclude we must show the final three terms on the right hand side of this are negative. For $\eps$ sufficiently small, we have
\begin{equation*}
 \frac{1}{|\log \eps|} \geq \frac{b^2+\rho_\star^2\frac{1}{\eps}\phi^R\left(\frac{\rho_\star}{\eps}\right)}{2b(R\eps-b)}
\end{equation*}
due to the specific form of $b$. Therefore
\begin{equation*}
 \frac{1}{|w|}\leq \frac{2b(R\eps-b)}{b^2+\rho_\star^2\frac{1}{\eps}\phi^R\left(\frac{\rho_\star}{\eps}\right)}
\end{equation*}
and so 
\begin{equation*}
 \frac{1}{|w|}\left(\left(b^2+\rho_\star^2\frac{1}{\eps}\phi^R\left(\frac{\rho_\star}{\eps}\right)\right) \frac{1}{|w|}-2b(R\eps-b)\right)\leq 0
\end{equation*}
as required.
 \end{proof}
 
This estimate is then used to prove the first aim of this section, that the removal of the impact parameters in the previous definition ensures that the short and long range evolutions exhibit the same collisional structure with high probability. We introduce the  notation here of $\omega=\{x_1,v_1,\dots,x_N,v_N\}$ to be the initial positions and velocities of the background particles. The initial conditions of the $i$th background particle are then denoted by $\omega_{i}$.
 \begin{lemma}
  Suppose that $\Phi\in \cR(\eps)$ with $\eps>0$ sufficiently small. Furthermore suppose that $R= \eps^{-\frac{1}{3+\gamma}}$, then we have
  \begin{multline*}
   \mathbb{P}\Big[x^{\eps,R} \text{ and } x^\eps \text{ have same collisions}\\\Big| \omega_{k+1},\dots,\omega_N,\,s.t.\,\,\forall s\in[0,T], \,\,|x^{\eps,R}-(x_i+sv_i)|>R\eps+2b(\eps)\Big] =1.
  \end{multline*}
 \end{lemma}

 \begin{proof}
 We aim to show that by restricting the impact parameters using the set $\cR(\eps)$ we ensure that the evolutions $x^{\eps,R}$ and $x^\eps$ encounter the same background. We prove by induction on the number of collisions already encountered.
 
 If one has encountered no collisions, then under the constraint that the background particles are at least $R\eps+2b(\eps)$ from $x^{\eps,R}$, by integrating the equations~\eqref{eq:partdens} we have
 \begin{equation*}
  |x^{\eps,R}(t)-x^\eps(t)|\leq N\,t\,\norm{(1-\Lambda^R)\nabla\phi}_{L^\infty}\le b(\eps)
 \end{equation*}
 and so the long range evolution does not encounter a near collision with any of the $N-n(\Phi)$ background particles not described in the tree $\Phi$.
 
 Now suppose that the short range evolution collides at time $t_1$. Again by Lemma~\ref{lem:SLdynerr}, we know that 
 \begin{equation*}
  |x^{\eps,R}(t)-x^\eps(t)|\leq b(\eps)
 \end{equation*}
 and we must ensure that the long range tagged particle also encounters a collision with this background. Since $\Phi \in \cR(\eps)$, the impact parameter of the collision is thus smaller than $R-b(\eps)(1+1/V_1(\eps))/\eps$ and so by an application of Lemma~\ref{lem:minrad}, we know that the minimum radius of the collision is smaller than $R\eps-b(\eps)$ thus ensuring the long range evolution has a near collision with this background particle.
 
 This then concludes the base case of the inductive argument. 
 The remainder of the argument is identical to the base case. We use Lemma~\ref{lem:SLdynerr} to estimate the error between the long and short range evolutions, before using Lemma~\ref{lem:minrad} to ensure that the long range evolution encounters the same near collision.
 \end{proof}
  
  It should be clear that the conditioning on the background particles in the previous lemma has probability $0$ in the limit $\eps\to 0$. Indeed, the conditioning forces 
  \begin{equation*}
   \inf_{t\in [0,T]}|x^{\eps,R}-x_s|\nin[R\eps-b(\eps)(1+1/V_1(\eps)),\,R\eps+2b(\eps)],
  \end{equation*}
  for all time $t\in [0,T]$. This then forces the initial positions and velocities of the background particles to lie outside a cylinder of size $\left(CT\,V_2(\eps)\,b(\eps))^2\right)^{N-n(\Phi)}$, which we observe tends to $0$ as $\eps\to 0$. 
    
To conclude the section, we are left to show that the restriction of $\cR(\eps)$ has small measure.
 \begin{lemma}\label{lem:measureR}
 For  $\phi$ an admissible long range potential with decay as in \eqref{eq:potdecay}, recall the sets $\cG(\eps)$ and $\cR(\eps)$ in Definitions~\ref{geps} and \ref{def:nongraz} respectively. Furthermore, suppose that $R(\eps)=\eps^{-1/(3+\gamma)}$. Then for $f^{\eps,R}$ the short range tagged particle density on $\cMT$, we have
  \begin{equation*}
   f^{\eps,R}_t(\cG(\eps)\setminus \cR(\eps))\to 0
  \end{equation*}
as $\eps \to 0$
 \end{lemma}
\begin{proof}
 We start by observing that, where $\lambda$ is the  Lebesgue measure on $\cMT$, we have for $\eps$ sufficiently small
 \begin{equation*}
 \begin{aligned}
  \lambda(\cG(\eps)\setminus &\cR(\eps)) \le V_2(\eps)\sum_{k=1}^{M(\eps)} \Bigg(T\,V_2(\eps)\,b(\eps)\left(1+\frac{1}{V_1(\eps)}\right)\Bigg)^k \\
  &\leq C\, b(\eps)(1+|\log\eps|^7)  \sum_{k=0}^\infty \left(b(\eps)(1+|\log \eps|^4)\right)^k
 \end{aligned}
 \end{equation*}
and since $b(\eps)= Ce^{-C(1/\eps)^{\gamma/(3+\gamma)}}$, the sum is finite, and the multiplying factor tends to $0$ as $\eps\to 0$.

Therefore, the set of trees we remove in $\cR(\eps)$ is measure $0$ in the limit. Since $f^{\eps,R}_t$  is absolutely continuous with respect to the Lebesgue measure, we also have 
\begin{equation*}
 f_t^{\eps,R}(\cG(\eps)\setminus \cR(\eps)) \to 0
\end{equation*}
as $\eps \to 0$ as required.
 \end{proof}
 
 \subsection{Weak Convergence of Particle Densities}

We now utilise these estimates in order to show that $f^{\eps,R}-f^\eps \to 0$ as $\eps\to 0$.
We aim to exploit the structure of the dynamics, namely that the long and short range dynamics are comparable where they encounter the same near collisions, and that this structure is displayed on a set of evolutions with probability one in the limit $\eps \to 0$.

\begin{lemma}\label{lem:wkconvpart}
Let $\phi$ be an admissible potential with decay as in equation~\eqref{eq:potdecay}. Let  $R=\eps^{-1/(3+\gamma)}$, and let $h\in C_b(\cU)$. Then for $f^\eps$ the phase space density of the tagged particle under equations~\eqref{eq:partdens}, and $f^{\eps,R}$ the tagged particle density for short range evolution on $\cMT$ given in Section~\ref{sec:MT}, we have
 \begin{equation*}
 \int_\cU h(x,v) \,f^{\eps}_t(x,v)\dd x \dd v-\int_{ \cMT}h(\Phi)\, f_t^{\eps,R}(\Phi)\dd \Phi\to 0
 \end{equation*}
 as $\eps \to 0$.
 \end{lemma}

\begin{proof}
To apply the previous lemmas, we first describe the set of background we remove to ensure that the long range evolution has the same collisions as the short range evolution. Recalling the notation $\omega$ for the initial positions and velocities of the background particles, we define the set $A$ by
\begin{equation*}
A=\left\{\omega : \inf_{t\in [0,T]}|x^{\eps,R}(t)-x_s(t)|\nin [R\eps-b(\eps)(1+1/V_1(\eps)),\,R\eps+2b(\eps)]\right\}.
 \end{equation*}
 We then split into the following
\begin{equation*}
\begin{aligned}
   \int_\cU h \,f^{\eps}_t\dd x \dd v-\int_{ \cMT}h\, f_t^{\eps,R}\dd \Phi&\le\int_\cU h \,f^{\eps}_t\left(1-\mathbb{P}[A]\right)\dd x \dd v \\&\hspace{1cm}+\int_\cU h \,f^{\eps}_t\,\mathbb{P}[A]\dd x \dd v-\int_{ \cR(\eps)}h\, f_t^{\eps,R}\dd \Phi\\&\hspace{5cm}-\int_{ \cG(\eps)\setminus \cR(\eps)}h\, f_t^{\eps,R}\dd \Phi.
\end{aligned}
\end{equation*}
 The final term of this expression tends to zero by an application of Lemma~\ref{lem:measureR}. We now treat the first term.
 
Estimating the probability of the set $A$ by estimating the size of the cylinder one must remove for each background particle to be outside $A$, we have
 \begin{equation*}
 \begin{aligned}
 \mathbb{P}[A^{C}]&\le
   CV_2(\eps)^{M(\eps)}\left((R\eps+2b(\eps))^2-(R\eps-b(\eps)(1+1/V_1(\eps)))^2\right)^{M(\eps)}\\
   &= CV_2(\eps)^{M(\eps)}b(\eps)^{2M(\eps)}.
\end{aligned}
   \end{equation*}
This then tends to $0$ as $\eps\to 0$ which then means that, since $h\,f^\eps$ is bounded, the first term tends to $0$ as $\eps \to 0$.

We finally analyse the middle expression. We claim that if this difference tends to $0$ for $h$ an indicator function, then we can conclude. Approximating a positive $h\in C_b$ by a sum of indicator functions, we can use Fatou's lemma to deduce the convergence of the densities tested against this $h$. Finally, for an arbitrary $h\in C_b$, we split into positive and negative parts and then approximate each with a sum of indicator functions, and then we can deduce weak convergence.

We thus assume for the remainder of the proof that $h=\1_\Omega$. Observe that, by Lemma~\ref{lem:SLdynerr}, the evolution $x^{\eps,R}$ for tree $\Phi\in \cG(\eps)$ and the evolution $x^\eps$  ending at $(x,v)$ with $N$ background particles lie within $b(\eps)$ of each other. This then gives an estimate on the spread of the supports of the probabilities, and so
\begin{equation*}
 \int_{\cR(\eps)\cap S_t^{\eps,R}(\Omega)}f_t^{\eps,R}(\Phi)\dd \Phi\leq \int_{\Omega_b}f^\eps(t,x,v)\,\mathbb{P}[A]\dd x \dd v
\end{equation*}
and 
\begin{equation*}
\int_\Omega f^\eps(t,x,v)\,\mathbb{P}[A]\dd x \dd v \leq \int_{\cR(\eps)\cap S^{\eps,R}_t (\Omega_b)}f^{\eps,R}_t(\Phi)\dd \Phi
\end{equation*}
where 
\begin{equation*}
 \Omega_b=\{(x,v)\in \cU : \,\, \exists\, (y,w)\in \Omega \text{ such that } |x-y|<b, \, |v-w|<b\}
\end{equation*}
is the set of points within $b$ of the set $\Omega$. 

We therefore obtain
\begin{multline}\label{eq:estimate}
\left|\int_\Omega f^\eps(t,x,v)\,\mathbb{P}[A]\dd x \dd v- \int_{\cR(\eps) \cap S^{\eps,R}_t(\Omega)}f_t^{\eps,R}(\Phi) \dd \Phi\right| \\\leq \max\left\{ \int_{\cR(\eps)\cap S^{\eps,R}_t(\Omega_b\setminus \Omega)}f_t^{\eps,R}(\Phi) \dd \Phi,\,\int_{\Omega_b\setminus \Omega}f^\eps(t,x,v)\dd x \dd v\right\}
\end{multline}
and we show that both the terms on the right hand side tend to $0$ as $\eps\to 0$.

Using the evolution equation for $f_t^{\eps,R}$ in equation~\eqref{eq:srpart} to provide an $L^\infty$ bound on $f^{\eps,R}$, we obtain
\begin{equation*}
  \int_{\cR(\eps) \cap S^{\eps,R}_t(\Omega_b\setminus \Omega)}f_t^{\eps,R}(\Phi) \dd \Phi
  \le (4R\,V_2(\eps))^{M(\eps)}\sum_{k=0}^{M(\eps)}  \lambda(S^{\eps,R}_t(\Omega_b\setminus \Omega)\cap \cMT_k)
\end{equation*}
and we calculate the size of these sets. The velocity constraint in $S^{\eps,R}_t(\Omega_b\setminus \Omega)$ enforces the initial velocity of the tagged particle to lie in a region of size at most $
\mathrm{diam}(\Omega)^{2}\,b $,
and the impact parameters and velocities lie in sets of size at most $R\,V_2(\eps)$. The time labels  lie in $[0,T]$ and so we obtain
\begin{equation*}
 \lambda(S^{\eps,R}_t(\Omega_b\setminus \Omega)\cap \cMT_{k})\leq C^k\,R^{k}\,V_2(\eps)^{k+2}\,b(\eps),
 \end{equation*}
 and using the summation formula for a geometric series results in

 \begin{equation*}
  \int_{S^{\eps,R}_t(\Omega_b\setminus \Omega)}f_t^{\eps,R}(\Phi) \dd \Phi\leq C\,V_2(\eps)^2\,b(\eps) \frac{\left(C\,R\,V_2(\eps)\right)^{2M(\eps)+1}-1}{C\,R\,V_2(\eps)-1},
 \end{equation*}
 and the exponential decay of $b(\eps)$ as $\eps \to 0$ ensures that this tends to $0$.
 
For the other term in \eqref{eq:estimate} we first must show that $f^\eps$ is in $L^\infty$. With $T^{-t}_{N}$ the solution operator for \eqref{eq:partdens} with $N$ background particles at $(x_i , \,v_i)$, we have
\begin{equation*}
  f^{\eps}(t,x,v)=  \int \prod_{i=1}^N g(v_i) f_0(T_N^{-t}(x,v)) \dd v_1\dots \dd v_N
  \leq \norm{f_0}_{L^\infty}
\end{equation*}
and since we assume $f_0\in L^\infty$, by taking the supremum over $x,v$ we have $f^\eps \in L^\infty$. We then estimate
\begin{equation*}
  \int_{\Omega_b\setminus \Omega}f^\eps(t,x,v)\dd x \dd v\leq \norm{f^\eps}_{L^\infty} \int_{\Omega_b\setminus \Omega}\dd x \dd v
  \leq C\norm{f_0}_{L^\infty} \,b(\eps)\,C \mathrm{diam}(\Omega)^{2},
\end{equation*}
and since $b\to 0$ exponentially, this term tends to $0$ as $\eps \to 0$, which concludes convergence.
\end{proof}
 
 \section{Comparison of Short Range and Long Range Boltzmann Equations}\label{sec:BEcomp}

 To conclude the proof of Theorem~\ref{thm:main}, we compare weak solutions of the linear Boltzmann equation for the potentials $\phi$ and $\phi^{R}$. We aim to show that $f^R\to f$, and to then conclude Theorem~\ref{thm:main}. 
 
 Recall that weak solutions for potential $\phi$ satisfy \eqref{eq:weaksol}, and that a weak solution of the linear Boltzmann equation for $\phi^{R}$ satisfies 
   \begin{equation}\label{eq:srweak}
 -\int_0^T\int_\cU  \left(\pd_t h +v\cdot \nabla_x h\right)\,f^R\dd x \dd v \dd t-\int_{\cU} f_0 \,h(0) \dd x \dd v= \int_0^T\langle  L^{R}(f^R) ,\,h\rangle \dd t
 \end{equation}
where
 \begin{equation*}
\langle L^{R}(f) ,\,h\rangle=\int_\cU\int_{\R^3}\int_{B_R(0)}(h(v^{\prime,R})-h(v))\,g_\star\,f\,|v_\star-v|\dd S\dd v_\star \dd x \dd v.
 \end{equation*}
 
Before concluding convergence of $f^R$ to $f$, we first compare the collision operators. It is necessary here to recall Lemma~\ref{lem:angdev}. From this, we have the following.

\begin{lemma}\label{lem:collest}
 Let $R>0$ and suppose that $\phi$ is an admissible long range potential such that there is a $\rho_2>0$ and $s>2$ such that for $\rho>\rho_2$ we have $\psi(\rho)\leq \rho^{-s}$.
 
Then for all $f$ with $\int_\cU(1+|v|^2)\,f\dd x\dd v<\infty$ and for all $h \in C_0^\infty([0,T)\times \cU)$ we have 
 \begin{equation*}
  |\langle L^R(f),h\rangle -\langle L(f),h\rangle|\leq C_1(R)\,\norm{\nabla_v h}_{L^\infty}\int_\cU(1+|v|^2)\,f\dd x \dd v
 \end{equation*}
 where $C_1(R)=o(1)$ as $R\to \infty$. Furthermore, we also have
 \begin{equation*}
  |\langle L^R(f),h\rangle \rangle|\leq C_2\,\norm{\nabla_v h}_{L^\infty}\int_\cU(1+|v|^2)\,f\dd x \dd v
 \end{equation*}
 where $C_2 $ is independent of $R$.
\end{lemma}
\begin{proof}
Set $\eta=\frac{1}{\log R}$ and then by using the triangle inequality, the Lipschitz nature of $h$, and by splitting the integration over $v_\star$ into $B_\eta(v)$ and $\R^3\setminus B_\eta(v)$, we can estimate the difference by
\begin{equation*}
\begin{aligned}
 |\langle L^R(f),h\rangle -&\langle L(f),h\rangle|\\ &\leq C\norm{\nabla_v h}_{L^\infty} \Bigg(\int_\cU\int_{\R^3\setminus B_\eta(v)}\int_{B_R}|v^{\prime,R}-v'|\,g_\star\,f\,|v_\star-v|\dd S\dd v_\star \dd x \dd v\\&\hspace{1cm}+\int_\cU\int_{B_\eta(v)}\int_{B_R}|v^{\prime,R}-v'|\,g_\star\,f\,|v_\star-v|\dd S\dd v_\star \dd x \dd v\\&\hspace{2cm}+\int_\cU\int_{\R^3\setminus B_\eta(v)}\int_{\cS\setminus B_R}|v'-v|\,g_\star\,f\,|v_\star-v|\dd S\dd v_\star \dd x \dd v\\&\hspace{3cm}+\int_\cU\int_{B_\eta(v)}\int_{\cS\setminus B_R}|v'-v|\,g_\star\,f\,|v_\star-v|\dd S\dd v_\star \dd x \dd v\Bigg).
 \end{aligned}
\end{equation*}
On the terms with  $v_\star \nin B_\eta (v)$ we use the estimate on the difference of scattering angles in Lemma~\ref{lem:angdev} and the estimate, for $r>R$ that 
\begin{equation*}
 |v'-v|\leq C\theta^R\,|v_\star-v|\leq  \frac{C}{1+\eta^2r^s} |v_\star-v|
\end{equation*}
to estimate the integrals above outside of $B_\eta(v)$  by
\begin{multline*}
C\left( \int_0^{R-1-1/\eta}\frac{r\,\kappa(r,R)\dd r}{\eta^2}+ \int_{R-1-1/\eta}^\infty\frac{C\,r\dd r}{1+\eta^2r^s}\right)\\\times\int_{\R^3}(1+|v_\star|^2)\,g_\star\dd v_\star\int_\cU(1+|v|^2)\,f(v) \dd x \dd v
\end{multline*}
where we observe that the estimates on the scattering angles reduce the triple integration into a product of three integrals. 

On $B_\eta(v)$, using the method of proof in Lemma~\ref{lem:angdev} we obtain with $w=v_\star-v$ the inequality
\begin{equation*}
 |\theta(r,w)-\theta^{R}(r,w)|\leq \begin{cases}\frac{C}{1+ |w|^2\, r^{s}}& \text{for } r>R-1-1/|w|\\
  &\\\frac{C\,\kappa(r,R)}{|w|^2} &\text{for } r<R-1-1/|w|,  \end{cases}
\end{equation*}
and using this to estimate the differences $|v^{\prime,R}-v'|$ and $|v'-v|$ on the set $B_\eta(v)$, and then using the form of $\kappa$ in \eqref{eq:kappa}, one obtains the form of $C_1(R)$ as
\begin{multline*}
 C_1(R):=C\Bigg( \int_0^{R-1-1/\eta}\frac{r\,\kappa(r,R)\dd r}{\eta^2}+ \int_{R-1-1/\eta}^\infty\frac{r\dd r}{1+\eta^2r^s}\\+\frac{1}{R^{s-3/2}\,\log^{7/2}R }+\int_{0}^{R-1}\frac{r\dd r}{\left(1-\frac{r^2}{R^2}\right)R^s\log^3R}\\+\frac{1}{\log^3R}\int_{0}^\infty\frac{r\dd r}{1 +r^s}\Bigg)\int_{\R^3}(1+|v_\star|^2)g_\star\dd v_\star
\end{multline*}
from which one can easily see that this is $o(1)$ as $R\to \infty$. For the operator $L^R$, by splitting the integration over $v_\star$ into the regions $B_1$ and $B_1^C$, and  using 
\begin{equation*}
 \theta^R(r,|w|) \leq \frac{C}{1+|w|^2 r^s},
\end{equation*}
we obtain the stated estimate on $L^R$, with
\begin{equation*}
 C_2= C\int_{\R^3}(1+|v_\star|^2)g_\star\dd v_\star\int_0^\infty \frac{r}{1+r^s}\dd r
\end{equation*}
for some $C$ depending only on the potential $\phi$.
\end{proof}

We use the estimate on the operator $L^R$ to extract a convergent subsequence of $\{f^R\}_{R\in (1,\infty)}$, and then show this limit satisfies \eqref{eq:weaksol} using the estimate on $L^R-L$.

\begin{lemma}\label{lem:BEcomp}
Suppose that $\phi$ is an admissible long range potential with decay as in equation~\eqref{eq:potdecay}. Then the sequence $\{f^R\}_{R\in (1,\infty)}$ has a convergent subsequence, and the limit is a weak solution of the linear Boltzmann equation associated to the potential $\phi$.
\end{lemma}
 \begin{proof}
 Firstly observe that the estimate on $L^R$ in Lemma~\ref{lem:collest} shows that the set $\{f^R\}$ is bounded uniformly in $R$ in $L^\infty$. By Banach Alaoglu this sequence converges weak-$\star$ in $L^\infty$, up to a subsequence, to a function $f\in L^\infty$. We now show that the limit $f$ is a weak solution, namely that it satisfies equation~\eqref{eq:weaksol}.

The uniform bound on the operators $L^R$ in Lemma~\ref{lem:collest} further ensures that
\begin{equation*}
 \int (1+|v|^2)\,f^R(t,x,v)\dd x \dd v< C<\infty
\end{equation*}
and we can thus pass to the limit in the term on the left hand side to obtain that 
\begin{equation}\label{eq:limitnorm}
 \int (1+|v|^2 )f=\int (1+|v|^2)\lim_{R\to\infty} f^R=\lim_{R\to\infty}\int (1+|v|^2)\,f^R<\infty.
\end{equation}

We now need to show that $f$ satisfies equation \eqref{eq:weaksol} for any suitable test function $h$, namely
\begin{equation*}
    -\int_0^T\int_\cU  \left(\pd_t h +v\cdot \nabla_x h\right)f\dd x \dd v \dd t-\int_{\cU} f_0\, h(0) \dd x \dd v= \int_0^T\langle L(f) ,h\rangle \dd t.
   \end{equation*}
Since $f^R$ is a weak solution of the linear Boltzmann equation for $\phi^R$, we can pass to the limit in the left hand side of equation~\eqref{eq:srweak} to obtain
   \begin{equation*}
    \int_0^T\int_\cU  \left(\pd_t h +v\cdot \nabla_x h\right)\, f^R\dd x \dd v \dd t \to \int_0^T\int_\cU  \left(\pd_t h +v\cdot \nabla_x h\right)\,f\dd x \dd v \dd t.
   \end{equation*}
We then observe that
\begin{equation*}
 \begin{aligned}
  \int_0^T \langle  L^R( f^R),h\rangle-\langle L(f),h\rangle \dd t &= \int_0^T \langle   L^R(  f^R-f),h\rangle \dd t+ \int_0^T \langle L^R(f)-L(f),h\rangle \dd t \\&\leq C_2\,\norm{\nabla_v h}_{L^\infty}\,\int_0^T\int_\cU(1+|v|^2)\,(f^R-f)\dd x \dd v \dd t\\&\hspace{1.5cm}+C_1(R)\,\norm{\nabla_v h}_{L^\infty}\,\int_0^T\int_\cU(1+|v|^2)\,f\dd x \dd v \dd t 
 \end{aligned}
\end{equation*}
where the estimates are from Lemma~\ref{lem:collest}. Since $C_2$ is bounded in $R$, equation~\eqref{eq:limitnorm} ensures the first term tends to $0$, and since $C_1=o(1)$ this ensures that the second term tends to $0$ as $R\to \infty$, thus showing $f$ is indeed a weak solution of equation \eqref{eq:weaksol}. 
\end{proof}

\begin{proof}{\bf (of Theorem~\ref{thm:main})}
We aim to conclude that $f^\eps \stackrel{\star}{\rightharpoonup} f$ in $L^\infty$ as $\eps \to 0$. We introduce $R= \eps^{-1/(3+\gamma)}$ and define the probability densities $f^{\eps,R}$ and $P^R$ on $\cMT$ as in Section~\ref{sec:MT}. We then write, for $h \in C^\infty_0(\cU) $ a test function,
 \begin{multline*}
  \int_\cU (f^\eps-f )\,h\dd x \dd v\leq \left|\int_\cU f^\eps(x,v)\,h\dd x \dd v  -\int_{\cMT}h(\Phi)\,f_t^{\eps,R} (\Phi)\dd \Phi \right| \\+\left|\int_\cMT h(\Phi)\, f_t^{\eps,R}(\Phi)\dd \Phi-\int_\cMT h(\Phi)\, P^R_t(\Phi)\dd \Phi\right|\\+ \left|\int_\cU (f^R-f)\,h\dd x \dd v\right|,
 \end{multline*}
 and we analyse each of these terms in the limit $\eps \to 0$.
 
 Firstly, Lemma~\ref{lem:wkconvpart} ensures that, with $R=\eps^{-\frac{1}{3+\gamma}}$, the first term converges to $0$ as $\eps\to 0$.
 For the second term, by following the proof of Lemma~\ref{srdensity}, one can see that by choosing $R\leq \eps^{-1/(3+\gamma)}$, the terms in $\rho_t^\eps(\Phi)$ still tend to $0$. We then conclude the convergence in total variation of $f^{\eps,R}$ to $P^R$ in the limit $\eps \to 0$.
 
 Finally, since $R\to \infty$ and the potential $\phi$ satisfies the decay \eqref{eq:potdecay}, which in particular implies that $\psi(\rho) \leq \rho^{-s}$ for $s>2$, we can apply Lemma~\ref{lem:BEcomp} to conclude that this term also tends to $0$ as $\eps \to 0$.
 This thus concludes the proof of the theorem. 
\end{proof}

\addcontentsline{toc}{section}{References}
\small{\bibliographystyle{plain2}
\bibliography{biblio-paper}}
\end{document}